\documentclass{amsart}

\makeatletter
\@namedef{subjclassname@2020}{%
  \textup{2020} Mathematics Subject Classification}
  
\makeatother 

\usepackage{amsthm,amssymb,amsfonts,latexsym,mathtools,thmtools,amsmath,lscape}
\usepackage[T1]{fontenc}
\usepackage{tikz-cd} 
\usepackage{enumitem} 
\usepackage{longtable}
\usepackage{array}
\usepackage{multirow}
\usepackage{mathrsfs} 
\usepackage{hyperref} 
\usepackage{hyperref}
\usepackage[all]{xy}
\usepackage{booktabs}
\hypersetup{
    colorlinks=true,
    linkcolor=blue,
    filecolor=blue,      
    urlcolor=blue,
    linktocpage=true
}

\newtheorem{theorem}{Theorem}[section]

\newtheorem{proposition}[theorem]{Proposition}

\theoremstyle{definition}
\newtheorem{definition}[theorem]{Definition}

\newtheorem{remark}[theorem]{Remark}
\newtheorem{question}[theorem]{Question}

\newtheorem{example}[theorem]{Example}

\theoremstyle{remark}

\numberwithin{equation}{section}

\begin{document}

\title[Smooth geometry of double regular algebras of type (14641)]{Smooth geometry of double extension \\ regular algebras of type (14641)}


\author{Andr\'es Rubiano}
\address{Universidad Nacional de Colombia - Sede Bogot\'a}
\curraddr{Campus Universitario}
\email{arubianos@unal.edu.co}
\thanks{}


\author{Armando Reyes}
\address{Universidad Nacional de Colombia - Sede Bogot\'a}
\curraddr{Campus Universitario}
\email{mareyesv@unal.edu.co}

\thanks{This work was supported by Faculty of Science, Universidad Nacional de Colombia - Sede Bogot\'a, Colombia [grant number 53880].}



\date{}

\dedicatory{Dedicated to Professor Oswaldo Lezama on the Occasion of His 68th Birthday}

\begin{abstract} 

In this paper, we prove that double extension regular algebras of type (14641) are not differentially smooth. 

\end{abstract}

\maketitle


\section{Introduction}

Ore \cite{Ore1931, Ore1933} introduced a kind of noncommutative polynomial rings which has become one of most basic and useful constructions in ring theory and noncommutative algebra. For an associative and unital ring $R$, an endomorphism $\sigma$ of $R$ and a $\sigma$-derivation $\delta$ of $R$, the {\em Ore extension} or {\em skew polynomial ring} of $R$ is obtained by adding a single generator $x$ to $R$ subject to the relation $xr = \sigma(r) x + \delta(r)$ for all $r\in R$. This Ore extension of $R$ is denoted by $R[x; \sigma, \delta]$. As one can appreciate in the literature, a lot of papers and books have been published concerning ring-theoretical, homological, geometrical properties and applications of these extensions (e.g. \cite{BrownGoodearl2002, BuesoTorrecillasVerschoren2003, Fajardoetal2020, Fajardoetal2024, GoodearlLetzter1994, GoodearlWarfield2004, McConnellRobson2001, Li2002, Rosenberg1995, SeilerBook2010} and references therein).

On the other hand, {\em Artin-Schelter regular algebras} defined by Artin and Schelter \cite{ArtinSchelter1987} are considered as noncommutative analogues of commutative polynomial rings due to its important role in noncommutative geometry. Since its introduction, there has been extensive research on the topic (see the excellent treatments carried out by Bellami et al. \cite{Bellamyetal2016} and Rogalski \cite{Rogalski2023}). Precisely, the classification of noncommutative projective 3-spaces or quantum $\mathbb{P}^3$s corresponds to the classification of Artin-Schelter regular algebras of global dimension four. With the aim of presenting new examples of Artin-Schelter regular algebras generated by four elements of degree one (Artin-Schelter regular algebras of dimension one and two that are generated by elements of degree one are well-known \cite[Examples 1.8 and 1.9]{Rogalski2023}, while of dimension three have been classified by Artin, Schelter, Tate and Van den Bergh \cite{ArtinSchelter1987, ArtinTateVandenBergh2007, ArtinTateVandenBergh1991}, c.f. \cite{Stephenson1996}), Zhang and Zhang \cite{ZhangZhang2008, ZhangZhang2009} introduced algebra extensions which they called {\em double Ore extensions} (or {\em double extensions} for short) as a natural generalization of the Ore extensions. In fact, from the definition of double extensions it is possible to appreciate some similarities to that of a two-step iterated Ore extensions. Nevertheless, there are no inclusions between the classes of all double extensions of an algebra and of all length two iterated Ore extensions of the same algebra \cite[Example 4.2 and Proposition 0.5(c)]{ZhangZhang2008}; Carvalho et al. \cite[Theorems 2.2 and 2.4]{Carvalhoetal2011} formulated necessary and sufficient conditions for a double extension to be presented as two-step iterated Ore extensions. 

As Zhang and Zhang asserted, rather than Ore extensions very few properties are known to be preserved under double extensions since some techniques used for Ore extensions are invalid for double extensions \cite[Section 4]{ZhangZhang2008}: \textquotedblleft double extensions seem much more difficult to study than Ore extensions [...] Many new and complicated constraints have to be posted in constructing a double extension. General ring-theoretic properties of double extensions are not known\textquotedblright\ \cite[Section 0]{ZhangZhang2008}. Some authors have contributed to the research on double extensions and its relations with algebraic structures such as Poisson, Hopf, Koszul and Calabi-Yau algebras (e.g. \cite{Li2022, LouOhWang2020, LuOhWangYu2018, LuWangZhuang2015, RamirezReyes2024, SuarezLezamaReyes2017, ZhuVanOystaeyenZhang2017}).

As we said above, in their paper \cite{ZhangZhang2009} Zhang and Zhang considered regular algebras $B$ of dimension four that are generated in degree one. By Lu et al. \cite{LuPalmieriWyZhang2007}, it is known that $B$ is generated by either two, or three, or four elements and the projective resolution of the trivial module $\Bbbk_B$ is given in \cite[Proposition 1.4]{LuPalmieriWyZhang2007} ($\Bbbk$ is an algebraically closed field). In the case that $B$ is generated by four elements, the projective resolution of the trivial module $\Bbbk_B$ is of the form
\begin{equation}\label{LISTResolution}
0 \xrightarrow{} B(-4) \xrightarrow{} B(-3)^{\oplus 4} \xrightarrow{} B(-2)^{\oplus 6} \to B(-1)^{\oplus 4} \to B \xrightarrow{} \Bbbk_B \xrightarrow{} 0.
\end{equation}

Due to the form of this resolution, Zhang and Zhang said that such an algebra is {\em of type} (14641). They classified all double extensions $R_P[y_1, y_2;\sigma]$ (with $\delta = 0$ and $\tau = (0, 0, 0)$) of type $(14641)$. Having in mind that Ore extensions and normal extensions of regular algebras of dimension three were studied by Le Bruyn et al. \cite{LeBruynSmithVandenBergh1996}, they omitted some of these from their classification, and hence their \textquotedblleft partial\textquotedblright\ classification consists of 26 families of regular algebras of type (14641) which are labeled by $\mathbb{A}, \mathbb{B}, \dotsc, \mathbb{Z}$. $\mathcal{LIST}$ denotes the class of all algebras in the families from $\mathbb{A}$ to $\mathbb{Z}$. In this class many of the double extensions $R_P[y_1, y_2;\sigma]$ are still Ore extensions; however, there might be non-zero $\delta$ and $\tau$ such that $R_P[y_1, y_2;\sigma]$ (with the same $(P, \sigma)$) is not an Ore extension. Zhang and Zhang's classification is basically the classification of $(P,\sigma)$ so that $R_P[y_1, y_2;\sigma, \delta, \tau]$ is not an Ore extension for possible $(\delta, \tau)$ \cite[p. 374]{ZhangZhang2009}. All details about this terminology and notions are given in Section \ref{DefinitionsandpreliminariesDOE}. 

Of interest in this paper, we refer to the notion of {\em differential smoothness}. Briefly, the study of {\em smoothness} of algebras goes back at least to Grothendieck's EGA \cite{Grothendieck1964}. The concept of a {\em formally smooth commutative} ({\em topological}) {\em algebra} introduced by him was extended to the noncommutative setting by Schelter \cite{Schelter1986}. An algebra is {\em formally smooth} if and only if the kernel of the multiplication map is projective as a bimodule. This notion arose as a replacement of a far too general definition based on the finiteness of the global dimension; Cuntz and Quillen \cite{CuntzQuillen1995} called these algebras {\em quasi-free}. Precisely, the notion of smoothness based on the finiteness of this dimension was refined by Stafford and Zhang \cite{StaffordZhang1994}, where a Noetherian algebra is said to be {\em smooth} provided that it has a finite global dimension equal to the homological dimension of all its simple modules. In the homological setting, Van den Bergh \cite{VandenBergh1998} called an algebra {\em homologically smooth} if it admits a finite resolution by finitely generated projective bimodules. The characterization of this kind of smoothness for the noncommutative pillow, the quantum teardrops, and quantum homogeneous spaces was made by Brzezi{\'n}ski \cite{Brzezinski2008, Brzezinski2014} and Kr\"ahmer \cite{Krahmer2012}, respectively.

Brzezi{\'n}ski and Sitarz \cite{BrzezinskiSitarz2017} defined other notion of smoothness of algebras, termed {\em differential smoothness} due to the use of differential graded algebras of a specified dimension that admits a noncommutative version of the Hodge star isomorphism, which considers the existence of a top form in a differential calculus over an algebra together with a string version of the Poincar\'e duality realized as an isomorphism between complexes of differential and integral forms. This new notion of smoothness is different and more constructive than the homological smoothness mentioned above. \textquotedblleft The idea behind the {\em differential smoothness} of algebras is rooted in the observation that a classical smooth orientable manifold, in addition to de Rham complex of differential forms, admits also the complex of {\em integral forms} isomorphic to the de Rham complex \cite[Section 4.5]{Manin1997}. The de Rham differential can be understood as a special left connection, while the boundary operator in the complex of integral forms is an example of a {\em right connection}\textquotedblright\ \cite[p. 413]{BrzezinskiSitarz2017}.

Several authors (e.g. \cite{Brzezinski2016, BrzezinskiElKaoutitLomp2010, BrzezinskiSitarz2017, DuboisViolette1988, DuboisVioletteKernerMadore1990, Karacuha2015, KaracuhaLomp2014, ReyesSarmiento2022}) have characterized the differential smoothness of algebras such as the quantum two - and three - spheres, disc, plane, the noncommutative torus, the coordinate algebras of the quantum group $SU_q(2)$, the noncommutative pillow algebra, the quantum cone algebras, the quantum polynomial algebras, Hopf algebra domains of Gelfand-Kirillov dimension two that are not PI, some 3-dimensional skew polynomial algebras, diffusion algebras in three generators, and noncommutative coordinate algebras of deformations of several examples of classical orbifolds such as the pillow orbifold, singular cones and lens spaces. An interesting fact is that some of these algebras are also homologically smooth in the Van den Bergh's sense. In particular, the differential smoothness of some families of Ore extensions has been investigated by some researchers \cite{Brzezinski2015, Brzezinski2016, BrzezinskiLomp2018, KaracuhaLomp2014}.

Having in mind the interest on the differential smoothness of Ore extensions and Zhang and Zhang's assertion about the few properties known to be preserved under double extensions, it is the purpose of this paper to investigate the differential smoothness of the double extension regular algebras of type (14641). As a matter of fact, since double extensions correspond to a construction of {\em bi-quadratic algebras} on four generators with PBW (Poincar\'e-Birkhoff-Witt) basis defined by Bavula \cite[p. 699]{Bavula2023}, and that the differential smoothness of bi-quadratic algebras on three generators with PBW basis was characterized by the authors in \cite{RubianoReyes2024DSBiquadraticAlgebras}, this paper is a sequel of the research of the smooth geometry of Bavula's algebras and two-step iterated Ore extensions from Brzezi{\'n}ski and Sitarz's point of view.

The paper is organized as follows. In Section \ref{DefinitionsandpreliminariesDOE} we review the key facts on double extension regular algebras of type (14641) in order to set up notation and render this paper self-contained. The corresponding list of these algebras is given in Tables \ref{FirstTableDOE}, \ref{SecondTableDOE}, \ref{ThirdTableDOE} and \ref{FourthTableDOE}. Next, Section \ref{DefinitionsandpreliminariesDSA} contains definitions and preliminaries on differential smoothness of algebras; we also set up notation necessary for the paper. In Section \ref{DICDOE} we prove our main result, Theorem \ref{DifferentiallySmoothDOE}, which says that none double extension of type (14641) is differentially smooth. Examples \ref{ZhangZhang2009Subcase4.1.1}, \ref{ZhangZhang2008Example4.2} and \ref{DoubleExtensionsDSYorN} show interesting facts on the differential smoothness of double extensions of differentially smooth algebras; these allow us to answer Question \ref{QuestionDSextendsDO}. Finally, we say some words about a future research.

Throughout the paper, $\mathbb{N}$ denotes the set of natural numbers including zero. The word ring means an associative ring with identity not necessarily commutative. $K$ and $\Bbbk$ denote a commutative ring with identity and an algebraically closed field, respectively. The term algebra means $\Bbbk$-algebra and $M_{r\times c}(R)$ denotes the ring of matrices of size $r\times c$ with entries in $R$.

\section{Double extension regular algebras of type (14641)}\label{DefinitionsandpreliminariesDOE}

We recall the definition of a double extension introduced by Zhang and Zhang \cite{ZhangZhang2008}. Since some typos ocurred in their papers \cite[p. 2674]{ZhangZhang2008} and \cite[p. 379]{ZhangZhang2009} concerning the relations that the data of a double extension must satisfy, we follow the corrections presented by Carvalho et al. \cite{Carvalhoetal2011}.

\begin{definition}[{\cite[Definition 1.3]{ZhangZhang2008}; \cite[Definition 1.1]{Carvalhoetal2011}}]\label{DoubleOreDefinition}
Let $R$ be a subalgebra of a $\Bbbk$-algebra $B$.
\begin{itemize}
    \item[\rm (a)] $B$ is called a {\it right double extension} of $R$ if the following conditions hold:
    \begin{itemize}
        \item[\rm (i)] $B$ is generated by $R$ and two new indeterminates $y_1$ and $y_2$;
        \item[\rm (ii)] $y_1$ and $y_2$ satisfy the relation
        \begin{equation}\label{Carvalhoetal2011(1.I)}
        y_2y_1 = p_{12}y_1y_2 + p_{11}y_1^2 + \tau_1y_1 + \tau_2y_2 + \tau_0,
        \end{equation}
        for some $p_{12}, p_{11} \in \Bbbk$ and $\tau_1, \tau_2, \tau_0 \in R$;
        \item[\rm (iii)] $B$ is a free left $R$-module with basis $\left\{y_1^{i}y_2^{j} \mid i, j \ge 0\right\}$.
        \item[\rm (iv)] $y_1R + y_2R + R\subseteq Ry_1 + Ry_2 + R$.
 \end{itemize}
    \item[\rm (b)] A right double extension $B$ of $R$ is called a {\em double extension} of $R$ if
    \begin{enumerate}
        \item [\rm (i)] $p_{12} \neq 0$;
        \item [\rm (ii)] $B$ is a free right $R$-module with basis $\left\{ y_2^{i}y_1^{j}\mid i, j \ge 0\right\}$;
        \item [\rm (iii)] $y_1R + y_2R + R = Ry_1 + Ry_2 + R$.
    \end{enumerate}
\end{itemize}
\end{definition}

\begin{remark}
As is well-known, two-step iterated Ore extensions of the form $R[y_1; \sigma_1, \delta_1][y_2; \sigma_2, \delta_2]$ are free left $R$-modules with a basis $\{y_1^{n_1} y_2^{n_2}\}_{n_1, n_2 \ge 0}$ \cite[Lemma 1.5]{ZhangZhang2008}. It can be seen that in general these extensions are not a (right) double extension in the sense of Definition \ref{DoubleOreDefinition} (a) because this kind of extensions might not have a relation of the form (\ref{Carvalhoetal2011(1.I)}). However, if $\sigma_2$ is chosen properly so that (\ref{Carvalhoetal2011(1.I)}) holds for these extensions, then all of them become a (right) double extension \cite[p. 2671]{ZhangZhang2008}. In this way, many double extensions are iterated Ore extensions and the condition in Definition \ref{DoubleOreDefinition}(a)(iii) is reasonable. As Zhang and Zhang \cite[p. 2671]{ZhangZhang2008} said, \textquotedblleft Our definition of a double extension is neither most general nor ideal, but it works very well in \cite{ZhangZhang2009}\textquotedblright. 
\end{remark}

Condition (a)(iv) from Definition \ref{DoubleOreDefinition} is equivalent to the existence of two maps
\[
\sigma = \begin{bmatrix}
    \sigma_{11} & \sigma_{12} \\ \sigma_{21} & \sigma_{22}
\end{bmatrix}: R\to M_{2\times 2}(R)\quad {\rm and}\quad \delta = \begin{bmatrix} \delta_1 \\ \delta_2  \end{bmatrix}: R\to M_{2\times 1}(R),
\]

such that
\begin{equation}\label{Carvalhoetal2011(1.II)}
    \begin{bmatrix}
        y_1 \\ y_2
    \end{bmatrix} r = \sigma(r) \begin{bmatrix}
        y_1 \\ y_2
    \end{bmatrix} + \delta(r) \quad {\rm for\ all}\ r\in R.
\end{equation}

If $B$ is a right double extension of $R$, we write $B = R_P[y_1, y_2;\sigma, \delta, \tau]$, where $P = (p_{12}, p_{11})$ with elements belonging to $\Bbbk$, $\tau = \{\tau_0, \tau_1, \tau_2\} \subseteq R$, and $\sigma, \delta$ are as above. $P$ is called a {\em parameter} and $\tau$ a {\em tail}, while the set $\{P, \sigma, \delta, \tau\}$ is said to be the {\em DE-data}. One of the particular cases of the double extensions is presented by Zhang and Zhang \cite[Convention 1.6.(c)]{ZhangZhang2008} as a {\it trimmed double extension}, for which $\delta$ is the zero map and $\tau = \{0, 0, 0\}$. We use the short notation $R_p[y_1, y_2; \sigma]$ to denote this subclass of extensions. This kind of double extensions will be important later.

For a right double extension $R_p[y_1, y_2; \sigma, \delta, \tau]$, all maps $\sigma_{ij}$ and $\delta_i$ are endomorphisms of the $\Bbbk$-vector space $R$. From \cite[Lemma 1.7]{ZhangZhang2008} we know that $\sigma$ must be a homomorphism of algebras, and $\delta$ is a $\sigma$-derivation in the sense that $\delta$ is $\Bbbk$-linear and satisfes $\delta(rr') = \sigma(r)\delta(r') + \delta(r)r'$ for all $r, r'\in R$. It is straightforward to see that if the matrix $\begin{bmatrix} \sigma_{11} & \sigma_{12} \\ \sigma_{21} & \sigma_{22} \end{bmatrix}$ is triangular, then both $\sigma_{11}$ and $\sigma_{22}$ are algebra homomorphisms. 


In the case that $\tau \subseteq \Bbbk$, the subalgebra of $R_P[y_1, y_2;\sigma, \delta, \tau]$ generated by $y_1$ and $y_2$ is the double extension $\Bbbk_P[y_1, y_2; \sigma', \delta', \tau']$, where $\sigma' = \sigma |_{\Bbbk}$ is the canonical embedding of $\Bbbk$ in $M_{2\times 2}(\Bbbk)$ and $\delta' = \delta |_{\Bbbk} = 0$ is the zero map. Carvalho et al. \cite[Proposition 1.2]{Carvalhoetal2011} proved that the latter is always an iterated Ore extension.

The next result characterizes double extensions.

\begin{proposition}[{\cite[Lemma 1.10 and Proposition 1.11]{ZhangZhang2008}; \cite[Proposition 1.5]{Carvalhoetal2011}}]\label{Carvalhoetal2011Proposition1.5}
Given a $\Bbbk$-algebra $R$, let $\sigma$ be a homomorphism from $R$ to $M_{2\times 2}(R)$, $\delta$ a $\sigma$-derivation from $R$ to $M_{2\times 1}(R)$, $P = (p_{12}, p_{11})$ a set of elements of $\Bbbk$, and $\tau = \{\tau_0, \tau_1, \tau_2\}$ a set of elements of $R$. Then, the associative $\Bbbk$-algebra $B$ generated by $R, y_1$ and $y_2$ subject to the relations {\rm (}\ref{Carvalhoetal2011(1.I)}{\rm )} and {\rm (}\ref{Carvalhoetal2011(1.II)}{\rm )}, is a right double extension if and only if the maps $\sigma_{ij}$ and $\rho_k$, $i\in \{1, 2\}, j, k \in \{0, 1, 2\}$, satisfy the six relations {\rm (}\ref{Carvalhoetal2011(1.III)}{\rm )} - {\rm (}\ref{Carvalhoetal2011(1.VIII)}{\rm )}  below, where $\sigma_{i0} = \delta_i$ and $\rho_k$ is a \underline{right} multiplication by $\tau_k$: 
\begin{equation}\label{Carvalhoetal2011(1.III)} 
    \sigma_{21} \sigma_{11} + p_{11}\sigma_{22}\sigma_{11} = p_{11}\sigma_{11}^2 + p_{11}^2 \sigma_{12}\sigma_{11} + p_{12}\sigma_{11}\sigma_{21} + p_{11}p_{12}\sigma_{12}\sigma_{21},
\end{equation}
\begin{align}    
    \sigma_{21} \sigma_{12} + p_{12} \sigma_{22} \sigma_{11} = &\ p_{11} \sigma_{11} \sigma_{12} + p_{11}p_{12}\sigma_{12}\sigma_{11} + p_{12}\sigma_{11}\sigma_{22} + p_{12}^2\sigma_{12}\sigma_{21}, \\
    \sigma_{22}\sigma_{12} = &\ p_{11} \sigma_{12}^2 + p_{12}\sigma_{12}\sigma_{22}, \\
    \sigma_{20} \sigma_{11} + \sigma_{21}\sigma_{10} + \underline{\rho_1 \sigma_{22}\sigma_{11}} = &\ p_{11} (\sigma_{10} \sigma_{11} + \sigma_{11}\sigma_{10} + \rho_{1}\sigma_{12}\sigma_{11}) \notag \\
    &\ + p_{12} (\sigma_{10} \sigma_{21} + \sigma_{11} \sigma_{20} + \rho_1 \sigma_{12} \sigma_{21}) + \tau_1 \sigma_{11} + \tau_2 \sigma_{21}, \\
    \sigma_{20} \sigma_{12} + \sigma_{22} \sigma_{10} + \underline{\rho_2 \sigma_{22} \sigma_{11}} = &\ p_{11} (\sigma_{10} \sigma_{12} + \sigma_{12}\sigma_{10} + \rho_{2} \sigma_{12} \sigma_{11}) \notag \\
    &\ +p_{12} (\sigma_{10} \sigma_{22} + \sigma_{12} \sigma_{20} + \rho_2 \sigma_{12} \sigma_{21}) +  \tau_{1} \sigma_{12} + \tau_2 \sigma_{22}, \\
    \sigma_{20} \sigma_{10} + \underline{\rho_0 \sigma_{22} \sigma_{11}} = &\ p_{11} (\sigma_{10}^2 + \rho_0 \sigma_{12} \sigma_{11}) + p_{12} (\sigma_{10} \sigma_{20} + \rho_0 \sigma_{12} \sigma_{21})\notag \\
    &\ + \tau_1 \sigma_{10} + \tau_2 \sigma_{10} + \tau_0 {\rm id}_R. \label{Carvalhoetal2011(1.VIII)}
\end{align}
\end{proposition}

\begin{remark}
\begin{enumerate}
    \item [\rm (i)] \cite[Remark 1.6]{Carvalhoetal2011} \label{RemarkCarvalho} Proposition \ref{Carvalhoetal2011Proposition1.5} implies the uniqueness, up to isomorphism, of a right double extension of $R$, with given $\sigma, \delta, P$ and $\tau$, provided such an extension exists. Indeed, assume $\overline{B} = R_P[y_1, y_2; \sigma, \delta, \tau]$ is a right double extension of $R$. Then, by \cite[Lemmas 1.7 and 1.10(b)]{ZhangZhang2008}, the data $\sigma, \delta, P$ and $\tau$ satisfy the conditions of Proposition \ref{Carvalhoetal2011Proposition1.5}. Let $B$ be as in this proposition. Then, there is an algebra homomorphism from $B$ to $\overline{B}$ which restricts to the identity on $R$ and maps $y_i\in B$ to the corresponding element $y_i \in \overline{B}$, $i = 1, 2$. Since $B$ is a free left $R$-module with basis $\left\{y_1^{i} y_2^{j} \mid i, j \ge 0\right\}$ and the same holds for $\overline{B}$, this map is an isomorphism, thus proving uniqueness. 
    
    \item [\rm (ii)] \cite[p. 2842]{Carvalhoetal2011} Let $B = R_P[y_1, y_2; \sigma, \delta, \tau]$ be a right double extension and suppose that $p_{12} \neq 1$. Then, from the ideas above, by choosing adequate generators $\overline{y}_i$ and (possibly) modifying the data $\sigma, \delta, \tau$ one can assume that $p_{11} = 0$. If $\overline{B} = R_P[\overline{y}_1, \overline{y}_2 ; \overline{\sigma}, \overline{\delta}, \overline{\tau}]$ is a right double extension with $p_{11} = 0$, then $\overline{B}$ has a natural filtration, given by setting ${\rm deg}\ R = 0$ and ${\rm deg}\ \overline{y}_1 = {\rm deg}\ \overline{y}_2 = 1$. It can be seen that in view of relations (\ref{Carvalhoetal2011(1.I)}) and (\ref{Carvalhoetal2011(1.II)}), that the associated graded algebra $G(B)$ is isomorphic to $R_P[\overline{y}_1, \overline{y}_2; \overline{\sigma}, \{0, 0, 0\}]$. 
\end{enumerate}
\end{remark}

Next, we recall the key results formulated by Carvalho et al. \cite{Carvalhoetal2011} about relations between double extensions and two-step iterated Ore extensions (c.f. Zhang and Zhang \cite[Proposition 3.6]{ZhangZhang2009}).

\begin{proposition}\label{Carvalhoetal2011Theorems2.2and2.4}
\begin{enumerate}
    \item [\rm (1)] \cite[Theorem 2.2]{Carvalhoetal2011} Let $R, B$ be $\Bbbk$-algebras such that $B$ is an extension of $R$. Assume $P = \{p_{12}, p_{11}\} \subseteq \Bbbk$, $\tau = \{\tau_0, \tau_1, \tau_2\} \subseteq R$, $\sigma$ is an algebra homomorphism from $R$ to $M_{2}(R)$ and $\delta$ is a $\sigma$-derivation from $R$ to $M_{2\times 1}(R)$. 
\begin{itemize}
    \item [\rm (a)] The following conditions are equivalent:
    \begin{itemize}
        \item $B = R_P[y_1, y_2; \sigma, \delta, \tau]$ is a right double extension of $R$ which can be presented as an iterated Ore extension $R[y_1;\sigma_1, d_1][y_2; \sigma_2, d_2]$.
        
        \item $B = R_P[y_1, y_2; \sigma, \delta, \tau]$ is right double extension of $R$ with $\sigma_{12} = 0$.
        
        \item $R[y_1;\sigma_1, d_1][y_2; \sigma_2, d_2]$ is and iterated Ore extensions such that 
        \begin{align*}
            \sigma_2(R) \subseteq (R)&, \ \ \ \sigma_2(y_1)= p_{12}y_1 + \tau_2, \\
            d_2(R) \subseteq Ry_1 + R&, \ \ \ d_2(y_1) = p_{11}y^2_1 + \tau_1y_1 + \tau_0,
        \end{align*}
        for some $p_{ij} \in \Bbbk$ and $\tau_i \in R$. The maps $\sigma, \delta, \sigma_i$ and $\delta_i$, $i = 1,2$, are related by
    \[
    \sigma=\begin{bmatrix} \sigma_{1} & 0 \\ \sigma_{21} & \sigma_{2}|_R \end{bmatrix},
    \ \ \ \delta(a)= \begin{bmatrix} d_1(a)\\  d_2(r) - \sigma_{21}(a)y_1 \end{bmatrix}, \ \ \  \text{for all} \ a \in R.
    \]
    \end{itemize}

\item[\rm (b)] If one of the equivalent statements from {\rm (}1{\rm )} holds, then $B$ is a double extension of $R$ if and only if $\sigma_1 = \sigma_{11}$ and $\sigma_2|_R = \sigma_{22}$ are automorphisms of $R$ and $p_{12} \neq 0$.
\end{itemize}

\item [\rm (2)] \cite[Theorem 2.4]{Carvalhoetal2011} Let $B = R_P[y_1, y_2; \sigma, \delta, \tau]$ be a right double extension of the $\Bbbk$-algebra $R$, where $P = \{p_{12}, p_{11}\} \in \Bbbk$, $\tau = \{\tau_2, \tau_1, \tau_0\} \subseteq R$, $\sigma: R \to M_2(R)$ is an algebra homomorphism and $\delta: R \to M_{2 \times 1}(R)$ is a $\sigma$-derivation. Then, $B$ can be presented as an iterated Ore extension $R[y_2; \sigma'_2, d'_2][y_1; \sigma'_1, d'_1]$ if and only if $\sigma_{21} = 0$, $p_{12} \neq 0$ and $p_{11} = 0$. In this case, $B$ is a double extension if and only if $\sigma'_2 = \sigma_{22}$ and $\sigma'_1|_R = \sigma_{11}$ are automorphisms of $R$.
\end{enumerate}
\end{proposition}

\begin{example}[{\cite[Example 4.1]{ZhangZhang2008}}]\label{ZhangZhang2008Example4.1}
Let $R = \Bbbk[x]$. As expected, there are so many different right double extensions of $R$ and if we assume that ${\rm deg}\ x = {\rm deg}\ y_1 = {\rm deg}\ y_2 = 1$, then all connected graded double extensions $\Bbbk[x]_P[y_1, y_2; \sigma, \delta, \tau]$ are regular algebras of global dimension three investigated by Artin and Schelter \cite{ArtinSchelter1987}. Let us see some particular situations to illustrate that the DE-data can be of different kinds. 

From relations (\ref{Carvalhoetal2011(1.I)}) and (\ref{Carvalhoetal2011(1.II)}) we have that 
\begin{align*}
    y_2 y_1 = &\ p_{12} y_1y_2 + p_{11}y_1^2 + axy_1 + bxy_2 + cx^2, \\
    y_1 x = &\ dxy_1 + exy_2 + fx^2, \quad {\rm and} \\
    y_2x = &\ gxy_1 + hxy_2 + ix^2.
\end{align*}

The six relations appearing in Proposition \ref{Carvalhoetal2011Proposition1.5} are equivalent to the fact that the overlap between the above three relations can be resolved. By using that $p_{12} \neq 0$, we may choose $P$ to be $(p, 0)$ or $(1, 1)$, while the invertibility of $\sigma$ is equivalent to the condition $dh - ge \neq 0$. 

Next, we present the list of defining relations of different possibilities for the algebra $B$.
\begin{enumerate}
    \item[\rm (i)] {\bf Algebra} $B^1:= B^1(p, a, b, c)$, where $a, b, c, p \in \Bbbk$ and $p \neq 0,\ b\neq 0, 1$.
    \begin{align}
        y_2 y_1 = &\ py_1 y_2 + \frac{bc}{1-b}(pb - 1)xy_1 + ax^2, \label{rel1B1} \\
        y_1 x = &\ bxy_1, \label{rel2B1} \quad {\rm and} \\
        y_2 x = &\ b^{-1}xy_2 + cx^2.\label{rel3B1}
    \end{align}

The homomorphism $\sigma$ is determined by $\sigma(x) = \begin{bmatrix}
    bx & 0 \\ 0 & b^{-1}x \end{bmatrix}$. In this case $\sigma_{12} = \sigma_{21} = 0$ and $\sigma_{11}$ and $\sigma_{22}$ are algebra automorphisms. The derivation is determined by $\delta(x) = \begin{bmatrix} 0 & cx^2 \end{bmatrix}$. The parameter $P$ is $(p, 0)$ and the tail is $\tau$ is $\left\{ \frac{bc}{1-b} (pb - 1)x, 0, ax^2 \right\}$. By using the linear transformation $y_2 \to y_2 + \frac{bc}{b - 1}x$ we get that $c = 0$. Thus, $B^{1}(1, p, a, b, c) \cong B^{1}(1, p, a, b, 0)$.

    \item[\rm (ii)] {\bf Algebra} $B^2:= B^2(a, b, c)$, where $a, b, c \in \Bbbk$ and $b\neq 0$.
    \begin{align*}
        y_2 y_1 = &\ -y_1 y_2 + ax^2, \\
        y_1 x = &\ b^{-1}xy_2 + cx^2, \quad {\rm and} \\
        y_2 x = &\ bxy_1 + (-bc) x^2.
    \end{align*}

The homomorphism $\sigma$ is determined by $\sigma(x) = \begin{bmatrix}
    0 & b^{-1}x \\ bx & 0 \end{bmatrix}$. It can be seen that $\sigma(x^n) = \begin{bmatrix}
    0 & b^{-1}x^n \\ bx^n & 0 \end{bmatrix}$ when $n$ is odd and $\sigma(x^n) = \begin{bmatrix}
    x^n & 0 \\ 0 & x^n \end{bmatrix}$ when $n$ is even. Note that in this case $\sigma_{11}$ and $\sigma_{22}$ are not algebra homomorphisms of $A$. The derivation is determined by $\delta(x) = \begin{bmatrix} cx^2 & -bcx^2 \end{bmatrix}$. The parameter $P$ is $(-1, 0)$ and the tail is $\tau$ is $\left\{ 0, 0, ax^2 \right\}$. Hence, $B^{2}(a, b, c) \cong B^{2}(ab^{-1}, 1, c)$.

    Note that for most $(a, b, c)$ (for example, $(a, b, c) = (1, 1, 0)$), the algebra $B^{2}(a, b, c)$ is not an Ore extension of any regular algebra of global dimension two, so $B^2(a, b, c)$ is not an iterated Ore extension of $\Bbbk[x]$ with ${\rm char}(\Bbbk) = 0$. Carvalho et al. \cite[Proposition 2.7]{Carvalhoetal2011} showed that this is not so in case the characteristic of the base field is two.
    
    \item[\rm (iii)] {\bf Algebra} $B^3:= B^3(a)$, where $a \in \Bbbk^{*}$.
    \begin{align*}
        y_2 y_1 = &\ y_1 y_2, \\
        y_1 x = &\ axy_1 + xy_2, \quad {\rm and} \\
        y_2 x = &\ axy_2.
    \end{align*}
    
The homomorphism $\sigma$ is determined by $\sigma(x) = \begin{bmatrix}
    ax & x \\ 0 & ax \end{bmatrix}$ and the derivation is zero. The parameter $P$ is $(1, 0)$ and the tail $\tau$ is $\{0, 0, 0\}$.

    \item[\rm (iii)] {\bf Algebra} $B^4:= B^4(a, b, c)$, where $a, b, c \in \Bbbk$ and $b\neq -1$.
    \begin{align*}
        y_2 y_1 = &\ y_1 y_2+ y_1^2 + axy_1 + \frac{b}{1 + b}xy_2 + cx^2, \\
        y_1 x = &\ xy_1 + bx^2, \\
        y_2 x = &\ (2 + 2b^{-1})xy_1 + xy_2.
    \end{align*}

    The homomorphism $\sigma$ is determined by $\sigma(x) = \begin{bmatrix}
        x & 0 \\ (2 + 2b^{-1})x & x \end{bmatrix}$. The derivation is determined by $\delta(x) = \begin{bmatrix} bx^2 \\ 0 \end{bmatrix}$. The parameter $P$ is $(1, 1)$ and the tail $\tau$ is $\left\{ax, \frac{b}{1+b}x, cx^2\right\}$.
\end{enumerate}

If we denote by $B$ any of the algebras $B^{i}$, then it can be seen that the element $x$ is normalizing in $B$ and $B / \langle x\rangle$ is a regular algebra of global dimension two. For this reason, these algebras are called {\em normal extensions}. Except for $B^2$, these algebras are also iterated Ore extensions of $\Bbbk[x]$. A detailed classification and general properties of these algebras can be found in \cite{ArtinSchelter1987, ArtinTateVandenBergh2007, ArtinTateVandenBergh1991}. 
\end{example}

As we said in the Introduction, Zhang and Zhang \cite{ZhangZhang2009} were interested only in regular algebras $B$ of dimension four that are generated in degree one, and in the case that $B$ is generated by four elements, the projective resolution of the trivial module $\Bbbk_B$ is of the form (\ref{LISTResolution}), that is, algebras {\em of type} (14641). The next proposition shows explicitly the relation of these algebras with Ore extensions.

\begin{proposition}[{\cite[Theorem 0.1]{ZhangZhang2009}}]
Let $B$ be a connected graded algebra generated by four elements of degree one. If $B$ is a double extension $R_P[y_1, y_2; \sigma, \tau]$ where $R$ is an Artin-Schelter regular algebra of dimension two, then the following assertions hold:
\begin{enumerate}
    \item [\rm (1)] $B$ is a strongly Noetherian, Auslander regular and Cohen-Macaulay domain.

    \item [\rm (2)] $B$ is of type {\rm (}14641{\rm )}. As a consequence, $B$ is Koszul.

    \item [\rm (3)] If $B$ is not isomorphic to an Ore extension of an Artin-Schelter regular algebra of dimension three, then the trimmed double extension $R_P[y_1, y_2; \sigma]$ is isomorphic to one of 26 families.
\end{enumerate}
\end{proposition}

In Tables \ref{FirstTableDOE}, \ref{SecondTableDOE}, \ref{ThirdTableDOE} and \ref{FourthTableDOE} we present the detailed list of the 26 families following the labels $\mathbb{A}, \mathbb{B}, \dotsc, \mathbb{Z}$ used in \cite{ZhangZhang2009}.

\subsection{Double extensions of \texorpdfstring{$\Bbbk_Q[x_1, x_2]$}{Lg}}

Since the base field $\Bbbk$ is algebraically closed, $R$ is isomorphic to $\Bbbk_q[x_1, x_2] = \mathcal{O}_q(\Bbbk)$ with the relation $x_2 x_1 = qx_1 x_2$ (note that $Q = (q, 0)$, the {\em Manin's plane}), or $\Bbbk_J[x_1, x_2] = \mathcal{J}(\Bbbk)$ with the relation $x_2x_1 = x_1x_2 + x_1^2$ (here, $Q = J = (1,1)$, the {\em Jordan's plane}) \cite[Theorem 1.4]{Shirikov2005} or \cite[Lemma 2.4]{ZhangZhang2009}. Manin's plane and Jordan's plane are the only regular algebras of global dimension two \cite[Examples 1.8 and 1.9]{Rogalski2023}.

In general, we will write $\Bbbk_Q[x_1, x_2]$, with $Q = (q_{12}, q_{11})$, and 
\[
\Bbbk_Q[x_1, x_2] = \Bbbk\{x_1, x_2\} / \langle x_2x_1 - q_{11}x_1^2 - q_{12}x_1x_2\rangle,
\]

and for the computation we set $Q$ to be either $(1, 1)$ or $(q, 0)$. Zhang and Zhang \cite[Section 3]{ZhangZhang2009} classified regular domains of dimension four of the form $R_P[y_1, y_2;\sigma]$ up to isomorphism (equivalently, to classify $(P, \sigma)$) up to some equivalence relation. They were interested in double extensions that are not iterated Ore extensions. Throughout this section we consider $R$ as just said.

Let $\sigma: R \to M_2(R)$ be a graded algebra homomorphism. Let
\begin{equation}\label{ZhangZhang2009(E3.0.1)}
\sigma_{ij}(x_s) = \sum_{t = 1}^2 a_{ijst} x_t,\quad {\rm for\ all}\ i, j, s = 1, 2 \ {\rm and} \ a_{ijst} \in \Bbbk.
\end{equation}

Using (\ref{ZhangZhang2009(E3.0.1)}), the relation (\ref{Carvalhoetal2011(1.II)}) of the algebra $R_P[y_1, y_2;\sigma]$ can be expressed as follows (note that in this case $\delta = 0$). If $r = x_1$ and $x_2$ in (\ref{Carvalhoetal2011(1.II)}), we get the following relations:
\begin{align}
y_1x_1 = &\ \sigma_{11}(x_1)y_1 + \sigma_{12}(x_1)y_2 \notag \\
= &\ a_{1111}x_1y_1 + a_{1112}x_2y_1 + a_{1211}x_1y_2 + a_{1212}x_2y_2, \label{ZhangZhangMR11} \\
y_1x_2 = &\ \sigma_{11}(x_2)y_1 + \sigma_{12}(x_2)y_2 \notag  \\
= &\ a_{1121}x_1y_1 + a_{1122}x_2y_1 + a_{1221}x_1y_2 + a_{1222}x_2y_2, \label{ZhangZhangMR12} \\
y_2x_1 = &\ \sigma_{21}(x_1)y_1 + \sigma_{22}(x_1)y_2 \notag  \\
= &\ a_{2111} x_1 y_1 + a_{2112}x_2y_1 + a_{2211}x_1y_2 + a_{2212}x_2y_2, \quad{\rm and} \label{ZhangZhangMR21} \\
y_2 x_2 = &\ \sigma_{21}(x_2)y_1 + \sigma_{22}(x_2)y_2 \notag \\
= &\ a_{2121}x_1y_1 + a_{2122}x_2y_1 + a_{2221}x_1y_2 + a_{2222}x_2y_2. \label{ZhangZhangMR22}
\end{align}

These four relations are called {\em mixing relations} between $x_i$ and $y_i$. The double extension $R_P[y_1, y_2; \sigma]$ also has two non-mixing relations given by
\begin{align}
    x_2 x_1 = &\ q_{12}x_1x_2 + q_{11}x_1^2, \quad {\rm and} \label{ZhangZhangNRx} \\
    y_2 y_1 = &\ p_{12}y_1y_2 + p_{11}y_1^2. \label{ZhangZhangNRy}
\end{align}

If we consider the matrices
\begin{equation}\label{ZhangZhang2009(E3.0.2)}
    \Sigma_{ij} := \begin{bmatrix}
        a_{ij11} & a_{ij12} \\ 
        a_{ij21} & a_{ij22}
    \end{bmatrix} \quad {\rm and} \quad \Sigma := \begin{bmatrix}
        \Sigma_{11} & \Sigma_{12} \\ 
        \Sigma_{21} & \Sigma_{22}
    \end{bmatrix} = \begin{bmatrix} 
    a_{1111} & a_{1112} & a_{1211} & a_{1212} \\
    a_{1121} & a_{1122} & a_{1221} & a_{1222} \\
    a_{2111} & a_{2112} & a_{2211} & a_{2212} \\
    a_{2121} & a_{2122} & a_{2221} & a_{2222} \\
    \end{bmatrix},
\end{equation}

since it is assumed that $\sigma$ is a graded algebra homomorphism, then $\sigma$ is uniquely determined by $\Sigma$. Another matrix related to $\Sigma$ is given in the following way. Let 
\begin{equation}\label{ZhangZhang2009(E3.0.2)'}
    M_{ij} := \begin{bmatrix}
        a_{11ij} & a_{12ij} \\
        a_{21ij} & a_{22ij} \end{bmatrix} \quad {\rm and} \quad M := \begin{bmatrix}
        M_{11} & M_{12} \\
        M_{21} & M_{22} 
    \end{bmatrix}.
\end{equation}

As we can see, the matrix $M$ is obtained by re-arranging the entries of $\Sigma$ and $\Sigma$ is invertible if and only if $M$ is invertible.

\begin{remark}[{\cite[p. 388-390]{ZhangZhang2009}}]
Using that $\sigma$ is an algebra homomorphism, for elements $i, j, f, g$ we have that
\begin{align*}
    \sigma_{ij}(x_f x_g) = &\ \sum_{p = 1} ^{2} \sigma_{ip}(x_f) \sigma_{pj}(x_g) \\
    = &\ \sum_{p, s, t = 1}^{2} (a_{ipfs}x_s)(a_{pjgt}x_t) \\ 
    = &\ \sum_{p, s, t = 1}^{2} (a_{ipfs} a_{pjgt})x_sx_t \\
    = &\ \left( \sum_{p = 1}^2 a_{ipf1} a_{pjg1} \right) x_1^{2} + \left( \sum_{p=1}^{2} a_{ipf1} a_{pjg2} \right) x_1 x_2 \\
    &\ + \left(\sum_{p = 1}^2 a_{ipf2} a_{pjg1}\right) x_2x_2 + \left(\sum_{p = 1}^2 a_{ipf2} a_{pjg2}\right)x_2^{2}.
\end{align*}

Since $x_2 x_1 = q_{11}x_1^2 + q_{12}x_1x_2$ in $R$, we obtain that 
\begin{align}
    \sigma_{ij}(x_fx_g) = &\ \left[\left( a_{i1f1}a_{1jg1} + a_{i2f_1}a_{2jg1}   \right) + q_{11}\left(a_{i1f2}a_{1jg1} + a_{i2f2}a_{2jg1}\right)\right]x_1^{2} \notag \\
    &\ + \left[\left( a_{i1f1}a_{1jg2} + a_{i2f1} a_{2jg2}\right) + q_{12}\left(a_{i1f2} a_{1jg1} + a_{i2f2} a_{2jg1}\right) \right] x_1x_1 \notag \\
    &\ + \left(a_{i1f2} a_{1jg2} + a_{i2f2} a_{2jg2}\right)x_2^{2}. \label{ZhangZhang2009(E3.0.3)}
\end{align}

Using the same relation and that each $\sigma_{ij}$ is a $\Bbbk$-linear map, it follows that
\begin{equation}\label{ZhangZhang2009(E3.0.4)}
\sigma_{ij}(x_2x_1) = q_{11}\sigma_{ij}(x_1x_1) + q_{12}\sigma_{ij}(x_1x_2) \quad {\rm for\ all}\ i, j = 1, 2.
\end{equation}

By (\ref{ZhangZhang2009(E3.0.3)}) the left-hand and the right-hand sides of (\ref{ZhangZhang2009(E3.0.4)}) can be expressed as polynomials of $x_1$ and $x_2$. If we compare the coefficients of $x_1^{2}, x_1x_2$ and $x_2^{2}$, respectively, then we obtain the following identities: the coefficients of $x_1^{2}$ of (\ref{ZhangZhang2009(E3.0.4)}) give rise to a constraint between coefficients
\begin{align}
   &\ \left(a_{i121} a_{1j11} + a_{i221} a_{2j11}\right) + q_{11}\left(a_{i122} a_{1j11} + a_{i222}a_{2j11}\right) \notag \\
    &\ \quad \ = q_{11} \left[ \left(a_{i111} a_{1j11} + a_{i211} a_{2j11}\right) + q_{11} \left(a_{i112} a_{1j11} + a_{i212}a_{2j11}\right) \right] \notag \\
    &\ \quad \quad \ + q_{12} \left[\left(a_{i111} a_{1j21} + a_{i211} a_{2j21}\right) + q_{11}\left(a_{i112} a_{1j21} + a_{i212}a_{2j21}\right)\right]. \label{ZhangZhang2009(C1ij)}
\end{align}

With respect to the coefficients of $x_1x_2$ of (\ref{ZhangZhang2009(E3.0.4)}), we get that
\begin{align}
&\ \left(a_{i121} a_{1j12} + a_{i221} a_{2j12}\right) + q_{12}\left(a_{i122} a_{1j11} + a_{i222} a_{2j11} \right) \notag \\
 &\ \quad \ = q_{11} \left[\left(a_{i111} a_{1j12} + a_{i211} a_{2j12}\right) + q_{12} \left(a_{i112} a_{1j11} + a_{i212} a_{2j11} \right)\right] \notag \\
  &\ \quad \quad \ + q_{12} \left[ \left( a_{i111} a_{1j22} + a_{i211} a_{2j22}\right) + q_{12} \left(a_{i112} a_{1j21} + a_{i212} a_{2j21}\right)\right]. \label{ZhangZhang2009(C2ij)}
\end{align}
\end{remark}

The coefficients of $x_2^2$ of (\ref{ZhangZhang2009(E3.0.4)}) satisfy the equality
\begin{align}
   &\ \left( a_{i122} a_{1j12} + a_{i222} a_{2j12} \right) \notag \\ 
     &\ \quad \ = q_{11} \left(a_{i112} a_{1j12} + a_{i212} a_{2j12}\right) + q_{12} \left( a_{i112} a_{1j22} + a_{i212} a_{2j22}\right). \label{ZhangZhang2009(C3ij)}
\end{align}

Applying relations (\ref{Carvalhoetal2011(1.I)}),  (\ref{Carvalhoetal2011(1.II)}) and all of them appearing in Proposition \ref{Carvalhoetal2011Proposition1.5} to the elements $r = x_1$ and $x_2$, we obtain more relations between the elements $a_{ijkl}$. For $i, f, g, s, t = 1, 2$, 
\begin{align}
    \sigma_{fg}(\sigma_{st}(x_i)) = &\ \sigma_{fg} \left( \sum_{w = 1}^2 a_{stiw} x_w \right) = \sum_{w = 1}^2 a_{stiw}\sigma_{fg}(x_w)\notag \\
    &\ \sum_{w=1}^{2} a_{stiw} \sum_{j = 1}^{2} a_{fgwj} x_j = \sum_{j=1}^{2} \left( a_{sti1} a_{fg1j} + a_{sti2} a_{fg2j}\right)x_j \notag \\
    = &\ \left( a_{sti1} a_{fg11} + a_{sti2} a_{fg21}\right) x_1 + \left( a_{sti1} a_{fg12} + a_{sti2} a_{fg22}\right)x_2. \label{ZhangZhang2009(E3.0.5)}
\end{align}

(Since $P = (p_{12}, p_{11})$, Zhang and Zhang will consider $P = (1, 1)$ or $(p, 0)$ in some computations). By (\ref{ZhangZhang2009(E3.0.5)}), the expressions (\ref{Carvalhoetal2011(1.I)}), (\ref{Carvalhoetal2011(1.II)}) and all of them appearing in Proposition \ref{Carvalhoetal2011Proposition1.5}, when applied to $x_i$ are equivalent to the following constraints:
\begin{align}
  &\  \left( a_{11i1} a_{211j} + a_{11i2} a_{212j}\right) + p_{11}\left(a_{11i2} a_{222j} \right) \notag \\
  &\ \quad \ = p_{11} \left( a_{11i1} a_{111j} + a_{11i2} a_{112j}\right) + p_{11}^2\left( a_{11i1} a_{121j} + a_{11i2} a_{122j} \right) \notag \\
   &\ \quad \quad \ + p_{12} \left(a_{21i1} a_{111j} + a_{21i2} a_{112j} \right) + p_{11} p_{12} \left( a_{21i1} a_{121j} + a_{21i2} a_{122j} \right), \label{ZhangZhang2009(C4ij)} \\
   &\ (a_{12i1} a_{211j} + a_{12i2} a_{212j}) + p_{12} \left( a_{11i1} a_{221j} + a_{11i2} a_{222j} \right) \notag \\
     &\ \quad \ = p_{11} \left( a_{22i1} a_{111j} + a_{12i2} a_{112j}\right) + p_{11}p_{12} \left( a_{11i1} a_{121j} + a_{11i2} a_{122j} \right) \notag \\
      &\ \quad \quad \ + p_{12} \left( a_{22i1} a_{111j} + a_{22i2} a_{112j} \right) + p_{12}^2 \left( a_{21i1} a_{121j} + a_{21i2} a_{122j} \right),  \label{ZhangZhang2009(C5ij)} \\
   &\  \left( a_{12i1} a_{221}j + a_{12i2} a_{222j} \right) \notag \\
       &\ \quad \ p_{11}\left( a_{12i1} a_{121j} + a_{12i2} a_{122j} \right) + p_{12}\left( a_{22i1} a_{121j} + a_{22i2} a_{122j} \right). \label{ZhangZhang2009(C6ij)}
\end{align}

As we can see, there is a symmetry between the first three $C$-constraints ((\ref{ZhangZhang2009(C1ij)}), (\ref{ZhangZhang2009(C2ij)}), (\ref{ZhangZhang2009(C3ij)})) and the last three $C$-constraints ((\ref{ZhangZhang2009(C4ij)}), (\ref{ZhangZhang2009(C5ij)}), (\ref{ZhangZhang2009(C6ij)})). The reasoning above is the content of the following proposition.

\begin{proposition}[{\cite[Proposition 3.1]{ZhangZhang2009}}]
\begin{enumerate}
    \item [\rm (1)] Let $R_P[y_1, y_2;\sigma]$ be a right double extension and suppose the data $\{\Sigma, P, Q\}$ are defined as above corresponding to the $\Bbbk$-linear basis $\{x_1, x_2, y_1, y_2\}$. Then the six equations {\rm (}\ref{ZhangZhang2009(C1ij)}{\rm)}, {\rm (}\ref{ZhangZhang2009(C2ij)}{\rm)}, {\rm (}\ref{ZhangZhang2009(C3ij)}{\rm)}, {\rm (}\ref{ZhangZhang2009(C4ij)}{\rm)}, {\rm (}\ref{ZhangZhang2009(C5ij)}{\rm)} and {\rm (}\ref{ZhangZhang2009(C6ij)}{\rm)} are satisfied. Further, ${\rm det}\ \Sigma \neq 0$ if and only if $R_P[y_1, y_2;\sigma]$ is a double extension.
    \item [\rm (2)] Suppose a matrix $\Sigma$ as in   and parameter sets $P = (p_{12}, p_{11})$ and $Q = (q_{12}, q_{11})$ with $p_{12} q_{12} \neq 0$ are given. If the six equations hold and ${\rm det}\ \Sigma \neq 0$, then the six relations   define a double extension $R_P[y_1, y_2;\sigma]$.
\end{enumerate}
\end{proposition}

The {\em System C} is the system of the six equations {\rm (}\ref{ZhangZhang2009(C1ij)}{\rm)}, {\rm (}\ref{ZhangZhang2009(C2ij)}{\rm)}, {\rm (}\ref{ZhangZhang2009(C3ij)}{\rm)}, {\rm (}\ref{ZhangZhang2009(C4ij)}{\rm)}, {\rm (}\ref{ZhangZhang2009(C5ij)}{\rm)} and {\rm (}\ref{ZhangZhang2009(C6ij)}{\rm)} together with ${\rm det}\  \Sigma \neq 0$. A {\em solution to System C} or a {\em C-solution} is a matrix $\Sigma$ with entries $a_{ijst}$ satisfying System C.

\begin{proposition}
Let $\Sigma$ be a C-solution and let $B = \left( \Bbbk_{Q}[x_1, x_2] \right)_P [y_1, y_2;\sigma]$ where $\sigma$ is determined by $\Sigma$. Let $0\neq h \in \Bbbk$.
\begin{enumerate}
    \item [\rm (1)] $B$ is a $\mathbb{Z}^2$-graded algebra. Let $\gamma: x_i \mapsto x_i,\ y_i \mapsto hy_i$. Then $\gamma$ extends to a graded automorphism of $B$.
    \item [\rm (2)] $h\Sigma$ is a $C$-solution. Let $\sigma'$ be the algebra automorphism determined by $h \Sigma$. Then $B':= \left( \Bbbk_Q[x_1, x_2] \right)_P[y_1, y_2;\sigma']$ is a graded twist of $B$ by $\gamma$ in the sense of Zhang \cite{Zhang1996}. 
\end{enumerate}
\end{proposition}

In general, the algebra $B$ and its twist $B^{\gamma}$ are not isomorphic as algebras. Nevertheless, both algebras have common properties since the category of graded $B$-mpdules is equivalent to the category of graded $B^{\gamma}$-modules \cite[Theorem 1.1]{Zhang1996}.

\begin{definition}[{\cite[Definition 3.4]{ZhangZhang2009}}]
\begin{enumerate}
    \item [\rm (i)] $\Sigma$ and $\Sigma'$ are {\em twist equivalent} if $\Sigma' = h\Sigma$ for some $0 \neq h \in \Bbbk$. In this case, $\Sigma'$ is called a {\em twist} of $\Sigma$. $\Sigma$ can be replace by its twists (without changing $Q$ and $P$) to obtain another double extension \cite[Lemma 3.3(b)]{ZhangZhang2009}.

    \item [\rm (ii)] $(\Sigma, Q, P)$ and $(\Sigma', Q', P')$ are {\em linearly equivalent} if there is a graded algebra isomorphism from $\left( \Bbbk_Q[x_1, x_2]\right)_P [y_1, y_2;\sigma]$ to $(\Bbbk_{Q'}[x_1', x_2'])_P' [y_1', y_2';\sigma']$ mapping $\Bbbk x_1 + \Bbbk x_2 \mapsto \Bbbk x_1' + \Bbbk x_2'$ and $\Bbbk y_1 + \Bbbk y_2 \mapsto \Bbbk y_1' + \Bbbk y_2'$. Using this isomorphism we can pull back $x_i'$ and $y_i'$ to the algebra $\left( \Bbbk_Q[x_1, x_2]\right)_P[y_1, y_2;\sigma]$; then we can assume that $\{x_1', x_2'\}$ (resp. $\{y_1', y_2'\}$) is another basis of $\{x_1, x_2\}$ (resp. $\{y_1, y_2\}$). In case $Q = Q'$ and $P = P'$, then we just say that $\Sigma$ and $\Sigma'$ are {\em linearly equivalent}.
    
    \item [\rm (iii)] $(\Sigma, Q, P)$ and $(\Sigma', Q', P')$ are {\em equivalent} if $(\Sigma, Q, P)$ and $(h \Sigma', Q', P')$ are linearly equivalent for some $0 \neq h\in \Bbbk$.
\end{enumerate}
\end{definition}

It is easy to show that twist equivalence and linear equivalence are equivalence relations. Also, the equivalence between  $(\Sigma, Q, P)$ and $(\Sigma', Q', P')$ is an equivalence relation.

Zhang and Zhang \cite{ZhangZhang2009} classified $R_P[y_1, y_2; \sigma]$ up to isomorphism (or even up to twist) by classifying $\Sigma$ up to (linear) equivalence.

\begin{proposition}\label{ZhangZhang2009Propositions3.6and3.7}
\begin{enumerate}
    \item [\rm (1)] \cite[Proposition 3.6]{ZhangZhang2009} Let $B$ be a double extension given by $(\Bbbk_Q[x_1, x_2])_P [y_1, y_2; \sigma, \tau, \delta]$.
\begin{enumerate}
    \item [\rm (a)] If $\Sigma_{12} = 0$, then $B$ is an iterated Ore extension.
    \item [\rm (b)] If $\Sigma_{21} = 0$ and $p_{11} = 0$, then $B$ is an iterated Ore extension.
\end{enumerate}

\item [\rm (2)] \cite[Proposition 3.7]{ZhangZhang2009} Let $B = \left( \Bbbk_Q[x_1, x_2] \right)_P [y_1, y_2;\sigma]$ be a trimmed double extension  where $\sigma$ is determined by the matrix $\Sigma$.
    \begin{enumerate}
    \item [\rm (a)] Considering $R': \Bbbk_P[y_1, y_2]$ as the subring and $\{x_1, x_2\}$ as the set of generators over $R'$, $B$ is a double extension $\left(\Bbbk_P[y_1, y_2]\right)_Q[x_1, x_2;\alpha]$ where $\alpha$ is determined by the matrix $M^{-1}$.
    
    \item [\rm (b)] If $M_{12} = 0$, then $B$ is an iterated Ore extension of $\Bbbk_P[y_1, y_2]$.
    
    \item [\rm (c)] If $M_{21} = 0$ and $q_{11} = 0$, then $B$ is an iterated Ore extension of $\Bbbk_P[y_1, y_2]$.
    \end{enumerate}
\end{enumerate}
\end{proposition}

\begin{remark}\label{ConditionssatisfiedDOOE}
\begin{enumerate}
    \item In the classification of double extensions, there are some of them that satisfy some condition of Proposition \ref{ZhangZhang2009Propositions3.6and3.7}. In fact, Zhang and Zhang comment that the algebra $\mathbb{A}$ satisfies $M_{12}=0$, and therefore it is an iterated Ore extension. However, under the choice of $\tau$ and $\delta$, the extension is not an iterated Ore extension for any regular algebra of dimension $2$ \cite[p. 395]{ZhangZhang2009}.

By Proposition \ref{ZhangZhang2009Propositions3.6and3.7}, we can see at the Tables \ref{FirstTableDOE}, \ref{SecondTableDOE}, \ref{ThirdTableDOE} and \ref{FourthTableDOE}, the following facts:
\begin{itemize}
\item All algebras $\mathbb{A}, \mathbb{B}, \dotsc, \mathbb{Z}$ satisfy condition (2) (a).

\item Algebra $\mathbb{N}$ satisfies condition $(1)(\text{a})$ when $f=0$.
    
\item Condition $(2)(\text{b})$ is satisfied by algebras $\mathbb{D}$, $\mathbb{G}$, $\mathbb{H}$, $\mathbb{K}$, $\mathbb{L}$, $\mathbb{Q}$, $\mathbb{X}$ and $\mathbb{Y}$. 

\item Algebra $\mathbb{M}$ satisfies condition $(2)(\text{c})$ when $f=0$.

\item Algebra $\mathbb{V}$ also satisfies $(2)(\text{c})$.

\item The other algebras do not satisfy any of the assumptions (except (2)(a)) of Proposition \ref{ZhangZhang2009Propositions3.6and3.7}.
\end{itemize}
\end{enumerate}
\end{remark}

\section{Differential smoothness of algebras}\label{DefinitionsandpreliminariesDSA}

We follow Brzezi\'nski and Sitarz's presentation on differential smoothness carried out in \cite[Section 2]{BrzezinskiSitarz2017} (c.f. \cite{Brzezinski2008, Brzezinski2014}).

\begin{definition}[{\cite[Section 2.1]{BrzezinskiSitarz2017}}]
\begin{enumerate}
    \item [\rm (i)] A {\em differential graded algebra} is a non-negatively graded algebra $\Omega$ with the product denoted by $\wedge$ together with a degree-one linear map $d:\Omega^{\bullet} \to \Omega^{\bullet +1}$ that satisfies the graded Leibniz rule and is such that $d \circ d = 0$. 
    
    \item [\rm (ii)] A differential graded algebra $(\Omega, d)$ is a {\em calculus over an algebra} $A$ if $\Omega^0 A = A$ and $\Omega^n A = A\ dA \wedge dA \wedge \dotsb \wedge dA$ ($dA$ appears $n$-times) for all $n\in \mathbb{N}$ (this last is called the {\em density condition}). We write $(\Omega A, d)$ with $\Omega A = \bigoplus_{n\in \mathbb{N}} \Omega^{n}A$. By using the Leibniz rule, it follows that $\Omega^n A = dA \wedge dA \wedge \dotsb \wedge dA\ A$. A differential calculus $\Omega A$ is said to be {\em connected} if ${\rm ker}(d\mid_{\Omega^0 A}) = \Bbbk$.
    
    \item [\rm (iii)] A calculus $(\Omega A, d)$ is said to have {\em dimension} $n$ if $\Omega^n A\neq 0$ and $\Omega^m A = 0$ for all $m > n$. An $n$-dimensional calculus $\Omega A$ {\em admits a volume form} if $\Omega^n A$ is isomorphic to $A$ as a left and right $A$-module. 
\end{enumerate}
\end{definition}

The existence of a right $A$-module isomorphism means that there is a free generator, say $\omega$, of $\Omega^n A$ (as a right $A$-module), i.e. $\omega \in \Omega^n A$, such that all elements of $\Omega^n A$ can be uniquely expressed as $\omega a$ with $a \in A$. If $\omega$ is also a free generator of $\Omega^n A$ as a left $A$-module, this is said to be a {\em volume form} on $\Omega A$.

The right $A$-module isomorphism $\Omega^n A \to A$ corresponding to a volume form $\omega$ is denoted by $\pi_{\omega}$, i.e.
\begin{equation}\label{BrzezinskiSitarz2017(2.1)}
\pi_{\omega} (\omega a) = a, \quad {\rm for\ all}\ a\in A.
\end{equation}

By using that $\Omega^n A$ is also isomorphic to $A$ as a left $A$-module, any free generator $\omega $ induces an algebra endomorphism $\nu_{\omega}$ of $A$ by the formula
\begin{equation}\label{BrzezinskiSitarz2017(2.2)}
    a \omega = \omega \nu_{\omega} (a).
\end{equation}

Note that if $\omega$ is a volume form, then $\nu_{\omega}$ is an algebra automorphism.

Now, we proceed to recall the key ingredients of the {\em integral calculus} on $A$ as dual to its differential calculus. For more details, see Brzezinski et al. \cite{Brzezinski2008, BrzezinskiElKaoutitLomp2010}.

Let $(\Omega A, d)$ be a differential calculus on $A$. The space of $n$-forms $\Omega^n A$ is an $A$-bimodule. Consider $\mathcal{I}_{n}A$ the right dual of $\Omega^{n}A$, the space of all right $A$-linear maps $\Omega^{n}A\rightarrow A$, that is, $\mathcal{I}_{n}A := {\rm Hom}_{A}(\Omega^{n}(A),A)$. Notice that each of the $\mathcal{I}_{n}A$ is an $A$-bimodule with the actions
\begin{align*}
    (a\cdot\phi\cdot b)(\omega)=a\phi(b\omega),\quad {\rm for\ all}\ \phi \in \mathcal{I}_{n}A,\ \omega \in \Omega^{n}A\ {\rm and}\ a,b \in A.
\end{align*}

The direct sum of all the $\mathcal{I}_{n}A$, that is, $\mathcal{I}A = \bigoplus\limits_{n} \mathcal{I}_n A$, is a right $\Omega A$-module with action given by
\begin{align}\label{BrzezinskiSitarz2017(2.3)}
    (\phi\cdot\omega)(\omega')=\phi(\omega\wedge\omega'),\quad {\rm for\ all}\ \phi\in\mathcal{I}_{n + m}A, \ \omega\in \Omega^{n}A \ {\rm and} \ \omega' \in \Omega^{m}A.
\end{align}

\begin{definition}[{\cite[Definition 2.1]{Brzezinski2008}}]
A {\em divergence} (also called {\em hom-connection}) on $A$ is a linear map $\nabla: \mathcal{I}_1 A \to A$ such that
\begin{equation}\label{BrzezinskiSitarz2017(2.4)}
    \nabla(\phi \cdot a) = \nabla(\phi) a + \phi(da), \quad {\rm for\ all}\ \phi \in \mathcal{I}_1 A \ {\rm and} \ a \in A.
\end{equation}  
\end{definition}

Note that a divergence can be extended to the whole of $\mathcal{I}A$, 
\[
\nabla_n: \mathcal{I}_{n+1} A \to \mathcal{I}_{n} A,
\]

by considering
\begin{equation}\label{BrzezinskiSitarz2017(2.5)}
\nabla_n(\phi)(\omega) = \nabla(\phi \cdot \omega) + (-1)^{n+1} \phi(d \omega), \quad {\rm for\ all}\ \phi \in \mathcal{I}_{n+1}(A)\ {\rm and} \ \omega \in \Omega^n A.
\end{equation}

By putting together (\ref{BrzezinskiSitarz2017(2.4)}) and (\ref{BrzezinskiSitarz2017(2.5)}), we get the Leibniz rule 
\begin{equation}
    \nabla_n(\phi \cdot \omega) = \nabla_{m + n}(\phi) \cdot \omega + (-1)^{m + n} \phi \cdot d\omega,
\end{equation}

for all elements $\phi \in \mathcal{I}_{m + n + 1} A$ and $\omega \in \Omega^m A$ \cite[Lemma 3.2]{Brzezinski2008}. In the case $n = 0$, if ${\rm Hom}_A(A, M)$ is canonically identified with $M$, then $\nabla_0$ reduces to the classical Leibniz rule.

\begin{definition}[{\cite[Definition 3.4]{Brzezinski2008}}]
The right $A$-module map 
$$
F = \nabla_0 \circ \nabla_1: {\rm Hom}_A(\Omega^{2} A, M) \to M
$$ is called a {\em curvature} of a hom-connection $(M, \nabla_0)$. $(M, \nabla_0)$ is said to be {\em flat} if its curvature is the zero map, that is, if $\nabla \circ \nabla_1 = 0$. This condition implies that $\nabla_n \circ \nabla_{n+1} = 0$ for all $n\in \mathbb{N}$.
\end{definition}

$\mathcal{I} A$ together with the $\nabla_n$ form a chain complex called the {\em complex of integral forms} over $A$. The cokernel map of $\nabla$, that is, $\Lambda: A \to {\rm Coker} \nabla = A / {\rm Im} \nabla$ is said to be the {\em integral on $A$ associated to} $\mathcal{I}A$.

Given a left $A$-module $X$ with action $a\cdot x$, for all $a\in A,\ x \in X$, and an algebra automorphism $\nu$ of $A$, the notation $^{\nu}X$ stands for $X$ with the $A$-module structure twisted by $\nu$, i.e. with the $A$-action $a\otimes x \mapsto \nu(a)\cdot x $.

The following definition of an \textit{integrable differential calculus} seeks to portray a version of Hodge star isomorphisms between the complex of differential forms of a differentiable manifold and a complex of dual modules of it \cite[p. 112]{Brzezinski2015}. 

\begin{definition}[{\cite[Definition 2.1]{BrzezinskiSitarz2017}}]
An $n$-dimensional differential calculus $(\Omega A, d)$ is said to be {\em integrable} if $(\Omega A, d)$ admits a complex of integral forms $(\mathcal{I}A, \nabla)$ for which there exist an algebra automorphism $\nu$ of $A$ and $A$-bimodule isomorphisms \linebreak $\Theta_k: \Omega^{k} A \to ^{\nu} \mathcal{I}_{n-k}A$, $k = 0, \dotsc, n$, rendering commmutative the following diagram:
\[
{\large{
\begin{tikzcd}
A \arrow{r}{d} \arrow{d}{\Theta_0} & \Omega^{1} A \arrow{d}{\Theta_1} \arrow{r}{d} & \Omega^2 A  \arrow{d}{\Theta_2} \arrow{r}{d} & \dotsb \arrow{r}{d} & \Omega^{n-1} A \arrow{d}{\Theta_{n-1}} \arrow{r}{d} & \Omega^n A  \arrow{d}{\Theta_n} \\ ^{\nu} \mathcal{I}_n A \arrow[swap]{r}{\nabla_{n-1}} & ^{\nu} \mathcal{I}_{n-1} A \arrow[swap]{r}{\nabla_{n-2}} & ^{\nu} \mathcal{I}_{n-2} A \arrow[swap]{r}{\nabla_{n-3}} & \dotsb \arrow[swap]{r}{\nabla_{1}} & ^{\nu} \mathcal{I}_{1} A \arrow[swap]{r}{\nabla} & ^{\nu} A
\end{tikzcd}
}}
\]

The $n$-form $\omega:= \Theta_n^{-1}(1)\in \Omega^n A$ is called an {\em integrating volume form}. 
\end{definition}

The algebra of complex matrices $M_n(\mathbb{C})$ with the $n$-dimensional calculus generated by derivations presented by Dubois-Violette et al. \cite{DuboisViolette1988, DuboisVioletteKernerMadore1990}, the quantum group $SU_q(2)$ with the three-dimensional left covariant calculus developed by Woronowicz \cite{Woronowicz1987} and the quantum standard sphere with the restriction of the above calculus, are examples of algebras admitting integrable calculi. For more details on the subject, see Brzezi\'nski et al. \cite{BrzezinskiElKaoutitLomp2010}. 

The following proposition shows that the integrability of a differential calculus can be defined without explicit reference to integral forms. This allows us to guarantee the integrability by considering the existence of finitely generator elements that allow to determine left and right components of any homogeneous element of $\Omega(A)$.

\begin{proposition}[{\cite[Theorem 2.2]{BrzezinskiSitarz2017}}]\label{integrableequiva} 
Let $(\Omega A, d)$ be an $n$-dimensional differential calculus over an algebra $A$. The following assertions are equivalent:
\begin{enumerate}
    \item [\rm (1)] $(\Omega A, d)$ is an integrable differential calculus.
    
    \item [\rm (2)] There exists an algebra automorphism $\nu$ of $A$ and $A$-bimodule isomorphisms $\Theta_k : \Omega^k A \rightarrow \ ^{\nu}\mathcal{I}_{n-k}A$, $k =0, \ldots, n$, such that, for all $\omega'\in \Omega^k A$ and $\omega''\in \Omega^mA$,
    \begin{align*}
        \Theta_{k+m}(\omega'\wedge\omega'')=(-1)^{(n-1)m}\Theta_k(\omega')\cdot\omega''.
    \end{align*}
    
    \item [\rm (3)] There exists an algebra automorphism $\nu$ of $A$ and an $A$-bimodule map $\vartheta:\Omega^nA\rightarrow\ ^{\nu}A$ such that all left multiplication maps
    \begin{align*}
    \ell_{\vartheta}^{k}:\Omega^k A &\ \rightarrow \mathcal{I}_{n-k}A, \\
    \omega' &\ \mapsto \vartheta\cdot\omega', \quad k = 0, 1, \dotsc, n,
    \end{align*}
    where the actions $\cdot$ are defined by {\rm (}\ref{BrzezinskiSitarz2017(2.3)}{\rm )}, are bijective.
    
    \item [\rm (4)] $(\Omega A, d)$ has a volume form $\omega$ such that all left multiplication maps
    \begin{align*}
        \ell_{\pi_{\omega}}^{k}:\Omega^k A &\ \rightarrow \mathcal{I}_{n-k}A, \\
        \omega' &\ \mapsto \pi_{\omega} \cdot \omega', \quad k=0,1, \dotsc, n-1,
    \end{align*}
    
    where $\pi_{\omega}$ is defined by {\rm (}\ref{BrzezinskiSitarz2017(2.1)}{\rm )}, are bijective.
\end{enumerate}
\end{proposition}

A volume form $\omega\in \Omega^nA$ is an {\em integrating form} if and only if it satisfies condition $(4)$ of Proposition \ref{integrableequiva} \cite[Remark 2.3]{BrzezinskiSitarz2017}.

The most interesting cases of differential calculi are those where $\Omega^k A$ are finitely generated and projective right or left (or both) $A$-modules \cite{Brzezinski2011}.

\begin{proposition}\label{BrzezinskiSitarz2017Lemmas2.6and2.7}
\begin{enumerate}
\item [\rm (1)] \cite[Lemma 2.6]{BrzezinskiSitarz2017} Consider $(\Omega A, d)$ an integrable and $n$-dimensional calculus over $A$ with integrating form $\omega$. Then $\Omega^{k} A$ is a finitely generated projective right $A$-module if there exist a finite number of forms $\omega_i \in \Omega^{k} A$ and $\overline{\omega}_i \in \Omega^{n-k} A$ such that, for all $\omega' \in \Omega^{k} A$, we have that 
\begin{equation*}
\omega' = \sum_{i} \omega_i \pi_{\omega} (\overline{\omega}_i \wedge \omega').
\end{equation*}

\item [\rm (2)] \cite[Lemma 2.7]{BrzezinskiSitarz2017} Let $(\Omega A, d)$ be an $n$-dimensional calculus over $A$ admitting a volume form $\omega$. Assume that for all $k = 1, \ldots, n-1$, there exists a finite number of forms $\omega_{i}^{k},\overline{\omega}_{i}^{k} \in \Omega^{k}(A)$ such that for all $\omega'\in \Omega^kA$, we have that
\begin{equation*}
\omega'=\displaystyle\sum_i\omega_{i}^{k}\pi_\omega(\overline{\omega}_{i}^{n-k}\wedge\omega')=\displaystyle\sum_i\nu_{\omega}^{-1}(\pi_\omega(\omega'\wedge\omega_{i}^{n-k}))\overline{\omega}_{i}^{k},
\end{equation*}

where $\pi_{\omega}$ and $\nu_{\omega}$ are defined by {\rm (}\ref{BrzezinskiSitarz2017(2.1)}{\rm )} and {\rm (}\ref{BrzezinskiSitarz2017(2.2)}{\rm )}, respectively. Then $\omega$ is an integral form and all the $\Omega^{k}A$ are finitely generated and projective as left and right $A$-modules.
\end{enumerate}
\end{proposition}

Brzezi\'nski and Sitarz \cite[p. 421]{BrzezinskiSitarz2017} asserted that to connect the integrability of the differential graded algebra $(\Omega A, d)$ with the algebra $A$, it is necessary to relate the dimension of the differential calculus $\Omega A$ with that of $A$, and since we are dealing with algebras that are deformations of coordinate algebras of affine varieties, the {\em Gelfand-Kirillov dimension} introduced by Gelfand and Kirillov \cite{GelfandKirillov1966, GelfandKirillov1966b} seems to be the best suited. Briefly, given an affine $\Bbbk$-algebra $A$, the {\em Gelfand-Kirillov dimension of} $A$, denoted by ${\rm GKdim}(A)$, is given by
\[
{\rm GKdim}(A) := \underset{n\to \infty}{\rm lim\ sup} \frac{{\rm log}({\rm dim}\ V^{n})}{{\rm log}\ n},
\]

where $V$ is a finite-dimensional subspace of $A$ that generates $A$ as an algebra. This definition is independent of choice of $V$. If $A$ is not affine, then its Gelfand-Kirillov dimension is defined to be the supremum of the Gelfand-Kirillov dimensions of all affine subalgebras of $A$. An affine domain of Gelfand-Kirillov dimension zero is precisely a division ring that is finite-dimensional over its center. In the case of an affine domain of Gelfand-Kirillov dimension one over $\Bbbk$, this is precisely a finite module over its center, and thus polynomial identity. In some sense, this dimensions measures the deviation of the algebra $A$ from finite dimensionality. For more details about this dimension, see the excellent treatment developed by Krause and Lenagan \cite{KrauseLenagan2000}.

After preliminaries above, we arrive to the key notion of this paper.

\begin{definition}[{\cite[Definition 2.4]{BrzezinskiSitarz2017}}]\label{BrzezinskiSitarz2017Definition2.4}
An affine algebra $A$ with integer Gelfand-Kirillov dimension $n$ is said to be {\em differentially smooth} if it admits an $n$-dimensional connected integrable differential calculus $(\Omega A, d)$.
\end{definition}

From Definition \ref{BrzezinskiSitarz2017Definition2.4} it follows that a differentially smooth algebra comes equipped with a well-behaved differential structure and with the precise concept of integration \cite[p. 2414]{BrzezinskiLomp2018}.

\begin{example}\label{Brzezinski2015ExamplesOE}
Several examples of noncommutative algebras have been proved to be differentially smooth. Let us briefly see at some of these that are related to Ore extensions.
\begin{enumerate}
\item [\rm (i)] The polynomial algebra $\Bbbk[x_1, \dotsc, x_n]$ has Gelfand-Kirillov dimension $n$ and the usual exterior algebra is an $n$-dimensional integrable calculus, which guarantees that $\Bbbk[x_1, \dotsc, x_n]$ is differentially smooth. 

\item [\rm (ii)] Related to Ore extensions, and under some mild and geometrically natural assumptions, Brzezi\'nski and Lomp proved that if $R$ and $S$ are algebras with integrable calculi $(\Omega R, d_R)$ and $(\Omega S, d_S)$, $\Omega R$ is a finitely generated projective right $R$-module and $\Omega$ is a finitely generated projective right $S$-module, then $(\Omega R \otimes \Omega S, d)$ is an integrable differential calculus for $R\otimes S$ \cite[Proposition 3.1]{BrzezinskiLomp2018}. With this, if $R$ and $S$ are differentially smooth algebras with respect to calculi which are finitely generated projective as right modules ${\rm GKdim}(R \otimes S) = {\rm GKdim}(R) + {\rm GKdim}(S)$, then the tensor product algebra $R\otimes S$ is differentially smooth \cite[Corollary 3.2]{BrzezinskiLomp2018}.

They also proved that if $R$ is an algebra with an integrable differential calculus $(\Omega R, d)$ such that $\Omega R$ is a finitely generated right $R$-module, for any automorphism $\sigma$ of $R$ that extends to a degree-preserving automorphism of $\Omega R$, which commutes with $d$, there exists an integrable differential calculus $(\Omega A, )$ on the skew polynomial ring $R[x;\sigma]$ and the Laurent skew polynomial ring $R[x^{\pm 1}; \sigma]$. If $R$ is differentially smooth with respect to $(\Omega R, d)$ and ${\rm GKdim}(A) = {\rm GKdim}(R) + 1$, then $A$ is also differentially smooth \cite[Theorem 4.1]{BrzezinskiLomp2018}. The following examples illustrates this situation.

For any non-zero $q\in \Bbbk^{*}$, let $A_q$ be the algebra generated by the indeterminates $x, y, z$ and relations $xy = yx,\ xz = qzy, \ yz = zx$, and let $B_q$ be the algebra generated by $x, y$ and invertible $z$ subject to the same relations. Since $A_q = \Bbbk[x, y][z;\sigma]$ and $B_q = \Bbbk[x, y][z^{\pm 1};\sigma]$ with $\sigma(x) = y$ and $\sigma(y) = qx$, then $A_q$ and $B_q$ are differentially smooth \cite[Example 4.6]{BrzezinskiLomp2018}. Let us see the details.

Note that the polynomial algebra $\Bbbk[x, y]$ is differentially smooth with the usual commutative differential calculus $\Omega (\Bbbk[x, y])$, i.e.,
\begin{align*}
       x dx = &\ dx x, & x dy = &\ dy x, & y dx = &\ dx y, \\ 
       y dy = &\ dy y, & dx dy = &\ - dy dx, & (dx)^2 = &\ (dy)^2 = 0. 
\end{align*}

The automorphism $\sigma$ extends to an automorphism of $\Omega (\Bbbk[x, y])$ by requesting it commute with $d$, that is, $\sigma(dx) = dy$ and $\sigma(dy) = q dx$.

Note that $\Omega (\Bbbk[x, y])$ is finitely generated as a right $\Bbbk[x, y]$-module and
\[
{\rm GKdim}(A_q) = {\rm GKdim}(B_q) = 3 = {\rm GKdim}(\Bbbk[x, y]) + 1,
\]

whence $A_q$ and $B_q$ are differentially smooth.

By using a similar reasoning \cite[Corollary 4.9]{BrzezinskiLomp2018}, it can be seen that the coordinate ring of the so called {\em quantum affine} $n$-{\em space}, that is, the algebra generated by the indeterminates $x_1, \dotsc, x_n$ subject to the relations $x_j x_i = q_{ij} x_i x_j$ for all $i < j$ with $q_{ij} \in \Bbbk^{*}$, is differentially smooth (c.f. \cite[Corollary 6 and Theorem 9]{KaracuhaLomp2014}).

\item [\rm (iii)] Brzezi{\'n}ski \cite{Brzezinski2015} characterized the differential smoothness of skew polynomial rings of the form $\Bbbk[t][x; \sigma_{q, r}, \delta_{p(t)}]$ where $\sigma_{q, r}(t) = qt + r$, with $q, r \in \Bbbk,\ q\neq 0$, and the $\sigma_{q, r}-$derivation $\delta_{p(t)}$ is defined as
\[
\delta_{p(t)} (f(t)) = \frac{f(\sigma_{q, r}(t)) - f(t)}{\sigma_{q, r}(t) - t} p(t),
\]

for an element $p(t) \in \Bbbk[t]$. $\delta_{p(t)}(f(t))$ is a suitable limit when $q = 1$ and $r = 0$, that is, when $\sigma_{q, r}$ is the identity map of $\Bbbk[t]$. In this way, the defining relation of the Ore extension $\Bbbk[t][x; \sigma_{q, r}, \delta_{p(t)}]$ is given by
\[
xt  = qtx + rx + p(t).
\]

For the maps
\begin{equation}\label{Brzezinski2015(3.4)}
\nu_t(t) = t,\quad \nu_t(x) = qx + p'(t)\quad {\rm and}\quad \nu_x(t) = \sigma_{q, r}^{-1}(t),\quad \nu_x(x) = x,
\end{equation}

where $p'(t)$ is the classical $t$-derivative of $p(t)$, Brzezi{\'n}ski \cite[Lemma 3.1]{Brzezinski2015} showed that all of them simultaneously extend to algebra automorphisms $\nu_t$ and $\nu_x$ of $\Bbbk[t][x; \sigma_{q, r}, \delta_{p(t)}]$ only in the following three cases:
    \begin{enumerate}
        \item [\rm (a)] $q = 1, r = 0$ with no restriction on $p(t)$;
        
        \item [\rm (b)] $q = 1, r\neq 0$ and $p(t) = c$, $c\in \Bbbk$;
        
        \item [\rm (c)] $q\neq 1, p(t) = c\left( t + \frac{r}{q-1} \right)$, $c\in \Bbbk$ with no restriction on $r$.
    \end{enumerate}
    
In any of the cases {\rm (a) - (c)} we have that $\nu_x \circ \nu_t = \nu_t \circ \nu_x$. If the Ore extension $\Bbbk[t][x; \sigma_{q, r}, \delta_{p(t)}]$ satisfies one of these three conditions, then it is differentially smooth \cite[Proposition 3.3]{Brzezinski2015}.

Due to Brzezi{\'n}ski's result on the differential smoothness of $\Bbbk[t][x; \sigma_{q, r}, \delta_{p(t)}]$, we get that the algebras 
\begin{itemize}
\item The {\em polynomial commutative ring} $\Bbbk[x_1, x_2]$;
        
\item The {\em Weyl algebra} $A_1(\Bbbk) = \Bbbk\{x_1, x_2\} / \langle x_1x_2 - x_2x_1 - 1\rangle$;
        
\item The {\em universal enveloping algebra of the Lie algebra} $\mathfrak{n}_2 = \langle x_1, x_2\mid [x_2, x_1] = x_1\rangle$, that is, $U(\mathfrak{n}_2) = \Bbbk\{x_1, x_2\} / \langle x_2x_1 - x_1x_2 - x_1\rangle$, and
        
\item {\em Manin's plane} $\mathcal{O}_q(\Bbbk) = \Bbbk \{x_1, x_2\} / \langle x_2 x_1 - qx_1 x_2\rangle$, where $q\in \Bbbk\ \backslash\ \{0,1\}$, and

\item {\em Jordan's plane} $\mathcal{J}(\Bbbk)$ with the relation given by $xt = tx + t^2$, 
\end{itemize}

are all differentially smooth.

\item [\rm (iv)] Brzezi\'nski et al. \cite{BrzezinskiElKaoutitLomp2010} proved that the coordinate algebras of the quantum group $SU_q(2)$, the standard quantum Podle\'s and the quantum Manin plane are differentially smooth.
\end{enumerate}
\end{example}

\begin{remark}
As expected, there are examples of algebras that are not differentially smooth. Consider the commutative algebra $A = \mathbb{C}[x, y] / \langle xy \rangle$. A proof by contradiction shows that for this algebra there are no one-dimensional connected integrable calculi over $A$, so it cannot be differentially smooth \cite[Example 2.5]{BrzezinskiSitarz2017}.
\end{remark}

\section{Differential smoothness of regular algebras of type (14641)}\label{DICDOE}

This section contains the important result of the paper.

\begin{theorem}\label{DifferentiallySmoothDOE}
None double extension regular algebra of type {\rm (}14641{\rm )} is differentially smooth. 
\end{theorem}
\begin{proof}
Consider a double extension $B$ generated by the set of indeterminates $B := \{ z_{+}, z_{-}, z'_{+}, z'_{-}\}$. The proof is by contradiction. We divide it into two parts.  
\begin{itemize}
    \item [\rm (1)] Suppose that for the elements $z'_{+}$ and $z_{+}$ the commutation rule 
\begin{align}\label{firstDO}
    z'_{+}z_{+} = \alpha z_{+}z'_{+} +\beta z_{-}z'_{+}+\gamma z_{+}z'_{-} +\delta z_{-}z'_{-},\quad {\rm where}\ \alpha,\beta,\gamma, \delta \in \Bbbk^{\ast}, 
\end{align}

is satisfied. Suppose that $B$ has a first order differential calculus $\Omega A$ with $d: A \rightarrow \Omega^1A$ a derivation. By applying $d$ to (\ref{firstDO}) we get that 
\[
0 = d(z'_{+}z_{+}) - d(\alpha z_{+}z'_{+} + \beta z_{-} z'_{+} + \gamma z_{+}z'_{-} + \delta z_{-}z'_{-})
\]

Since $d$ is $\Bbbk$-linear, using the Leibniz's rule it follows that
\begin{align*}
0 = &\ dz'_{+} z_{+} + z'_{+} dz_{+} - \alpha dz_{+} z'_{+} -\alpha z_{+} dz'_{+} - \beta dz_{-}z'_{+} \\
    & \ - \beta z_{-}dz'_{+} - \gamma dz_{+}z'_{-} - \gamma z_{+}dz'_{-} - \delta dz_{-}z'_{-} - \delta z_{-}dz'_{-}
\end{align*}

Using Property \ref{BrzezinskiSitarz2017(2.2)}, the action of the module is written using automorphisms $\nu_{z_+}$, $\nu_{z'_+}$ and $\nu_{z'_{-}}$
\begin{align*}
    0 & = dz'_{+}z_{+}+dz_{+}\nu_{z_{+}}(z'_{+})-\alpha dz_{+}z'_{+}-\alpha dz'_{+}\nu_{z'_{+}}(z_{+})-\beta dz_{-}z'_{+}\\
    & \ \ - \beta dz'_{+}\nu_{z'_{+}}(z_{-})-\gamma dz_{+}z'_{-}-\gamma dz'_{-}\nu_{z'_{-}}(z_{+})-\delta dz_{-}z'_{-}-\delta dz'_{-}\nu_{z'_{-}}(z_{-})
\end{align*}

The terms that multiply the different differentials are put together to obtain
\begin{align*}
    0 &\ =dz'_{+}(z_{+}-\alpha \nu_{z'_{+}}(z_{+})-\beta \nu_{z'_{+}}(z_{-}))+dz_{+}(\nu_{z_{+}}(z'_{+})-\alpha z'_{+}-\gamma z'_{-})\\
    &\ \ - dz_{-}(\beta z'_{+}+\delta z'_{-})-dz'_{-}(\gamma \nu_{z'_{-}}(z_{+})+\delta \nu_{z'_{-}}(z_{-})).
\end{align*}
 
Since the differentials are generating elements for $\Omega^1A$ and all of them are equal to zero, the elements that multiply them are also zero. We obtain that 
\begin{align*}
    \nu_{z'_{+}}(\alpha z_{+}+\beta z_{-}) = &\ z_{+}, \\
    \nu_{z_{+}}(z'_{+}) = &\ \alpha z'_{+}+\gamma z'_{-}, \\
    \beta z'_{+}+\delta z'_{-} = &\ 0, \quad {\rm and} \\
    \nu_{z'_{-}}(\gamma z_{+}+\delta z_{-}) = &\ 0,
\end{align*}

which are satisfied only when $\beta = \gamma = \delta = 0$. This contradicts our initial assumption.

\item [\rm (2)] Suppose that in the algebra $B$ is satisfied the quadratic relation 
\begin{equation}\label{QuadraticrelationDSDOE}
\sum_{t,w\in B } c_{t,w} tw = 0,\quad {\rm where}\ c_{t, w} \in \Bbbk,
\end{equation}

Without loss of generality, notice that if the indeterminate $z_{+}$ appears only once in the list of pairs $(t,w)$, call $(\bar{t},\bar{w})$, then there is no first order calculus. More exactly, consider the set $B'=B^{2}\setminus \{(\bar{t},\bar{w})\}$. By applying the differential $d$ to the expression (\ref{QuadraticrelationDSDOE}) we get that
\begin{align*}
   0 & =d\left(\sum_{t,w\in B } c_{t,k} tk\right) 
\end{align*}
Since $d$ is $\Bbbk$-linear, using the Leibniz's rule and emphasizing leaving aside the terms $\bar{t}\bar{w}$, it follows that
\begin{align*}
   0 & = \sum_{(t,w)\in B'} c_{t,w} d(tw) + c_{\bar{t},\bar{w}}d(\bar{t}\bar{w})\\
    & = \sum_{(t,w)\in B'} c_{t,w} (dtw+tdw)+c_{\bar{t},\bar{w}} (d\bar{t}\bar{w}+\bar{t}d\bar{w})
\end{align*}
Again, using Property \ref{BrzezinskiSitarz2017(2.2)}, the action of the module is written using automorphisms $\nu_{w}$, and $\nu_{\bar{w}}$
\begin{align*}
    0 & = \left(\sum_{(t,w)\in B'} c_{t,w} (dtw+dw\nu_{w}(t))\right) + d\bar{t}c_{\bar{t},\bar{w}}\bar{w}+d\bar{w}c_{\bar{t},\bar{w}}\nu_{\bar{w}}(\bar{t}).
\end{align*}
The generator $z_{+}$ does not appear as any component of the pairs $(t,w)\in B'$. In the case that $\bar{t} = z_{+}$, we obtain that $c_{\bar{t},\bar{w}}\bar{w} = 0$, whence $\bar{w} = 0$, which cannot be happen. Now, if $\bar{w} = z_{+}$, then $c_{\bar{t},\bar{w}}\nu_{\bar{w}}(\bar{t})=0$, and therefore $\nu_{\bar{w}}(\bar{t})=0$, which contradicts that $\nu_{z_{+}}$ is an automorphism. 
\end{itemize}
\end{proof}

From Remark \ref{ConditionssatisfiedDOOE} and Theorem \ref{DifferentiallySmoothDOE} we obtain the following facts:
\begin{itemize}
    \item Double extensions satisfying relation {\rm (}\ref{firstDO}{\rm )} are $\mathbb{C}$, $\mathbb{F}$, $\mathbb{I}$, $\mathbb{S}$, $\mathbb{T}$, $\mathbb{U}$ (see Tables \ref{FirstTableDOE}, \ref{SecondTableDOE}, \ref{ThirdTableDOE} and \ref{FourthTableDOE}).
    \item Double extensions that satisfy the commutation rule (\ref{QuadraticrelationDSDOE}) are $\mathbb{A}$, $\mathbb{B}$, $\mathbb{D}$, $\mathbb{E}$, $\mathbb{G}$, $\mathbb{H}$, $\mathbb{J}$, $\mathbb{K}$, $\mathbb{L}$, $\mathbb{M}$, $\mathbb{N}$, $\mathbb{O}$, $\mathbb{P}$, $\mathbb{Q}$, $\mathbb{R}$, $\mathbb{V}$, $\mathbb{W}$, $\mathbb{X}$, $\mathbb{Y}$ and $\mathbb{Z}$ (see Tables \ref{FirstTableDOE}, \ref{SecondTableDOE}, \ref{ThirdTableDOE} and \ref{FourthTableDOE}).
\end{itemize}

Let us see three illustrative examples on the differential smoothness of double extensions of differentially smooth algebras. These allow us to give an answer to the following question:

\begin{question}\label{QuestionDSextendsDO}
Let $B$ be a double extension of $R$. If $R$ is differentially smooth, then is $B$ differentially smooth? 
\end{question}

\begin{example}[{\cite[Subcase 4.1.1]{ZhangZhang2009}}]\label{ZhangZhang2009Subcase4.1.1}
    Let $B = (\Bbbk_Q[x_1,x_2])_P[y_1, y_2; \sigma, \delta, \tau]$ be the right double extension generated by $x_1, x_2, y_1, y_2$ subject to the relations 
\begin{align*}
    x_2x_1 & = x_1x_2 + x_1^2, & y_2y_1 & = y_1y_2 + y_1^2,\\
    y_1x_1 & = fx_1y_1, & y_1x_2 & = gx_1y_1 + fx_2y_1,\\
    y_2x_1 & = hx_1y_1 + fx_1y_2, & y_2x_2 & = mx_1y_1 + hx_2y_1 + gx_1y_2 + fx_2y_2,
\end{align*}

where $f, g, h, m \in \Bbbk$ and $f \neq 0$. Note that in the relations
\[
    y_1x_1 = \sigma_{11}(x_1)y_1 + \sigma_{12}(x_1)y_2,\quad {\rm and} \quad y_1x_2 = \sigma_{11}(x_2)y_1 + \sigma_{12}(x_2)y_2,
\]

we have that $\sigma_{12} = 0$. In this case $B = R_P[y_1, y_2; \sigma, \delta, \tau]$ is a right double extension of $R = \Bbbk_Q[x_1,x_2]$. By Proposition \ref{Carvalhoetal2011Theorems2.2and2.4} (1), $B$ can be presented as the two-step iterated Ore extension $\Bbbk_Q[x_1,x_2][y_1;\sigma_1, d_1][y_2; \sigma_2, d_2]$. If we take the relation $y_2x_1 = hx_1y_1 + fx_1y_2$, we can see that it satisfies the condition (\ref{QuadraticrelationDSDOE}), it follows that $B$ is not differentially smooth.
\end{example}

\begin{example}[{\cite[Example 4.2]{ZhangZhang2008}}]\label{ZhangZhang2008Example4.2}
Consider $h\in \Bbbk^{*}$. Let $B(h)$ denote the graded algebra generated by $x_1, x_2, y_1$ and $y_2$ subject to the conditions
\begin{align*}
    x_2 x_1 = &\ - x_1 x_2, \\
    y_2 y_1 = &\ - y_1 y_2, \\
    y_1 x_1 = &\ h(x_1 y_1 + x_2y_1 + x_1y_2),  \\
    y_1 x_2 = &\ h (x_1 y_2), \\
    y_2 x_1 = &\ h(x_2y_1), \quad {\rm and} \\
    y_2 x_2 = &\ h(-x_2y_1 - x_1y_2 + x_2y_2).
\end{align*}

Following the notation in \cite{ZhangZhang2009}, the set of the last four relations is associated to the matrix
\[
\Sigma = h \begin{bmatrix}
    1 & 1 & 1 & 0 \\
    0 & 0 & 1 & 0 \\
    0 & 1 & 0 & 0 \\
    0 & -1 & -1 & 1
\end{bmatrix}.
\]

$B(h)$ is a quadratic connected graded algebra where ${\rm deg}\ x_1 = {\rm deg}\ x_2 = {\rm deg}\ y_1 = {\rm deg}\ y_2 = 1$. In fact, from \cite[Lemma 4.5]{ZhangZhang2008} we know that by considering $P = (-1, 0), \delta = 0, \tau = (0, 0, 0)$ and the algebra homomorphism 
\begin{align*}
  \sigma: A &\ \to M_2(A) \\  
  x_1 \mapsto &\ h \begin{bmatrix}
    x_1 + x_2 & x_1 \\ x_2 & 0
\end{bmatrix}, \\  
   x_2 \mapsto &\ h \begin{bmatrix}
    0 & x_1 \\ -x_2 & -x_1 + x_2
\end{bmatrix},
\end{align*}

then the data $\{P, \sigma, \delta, \tau \}$ satisfy the relations (\ref{Carvalhoetal2011(1.III)}) - (\ref{Carvalhoetal2011(1.VIII)}) for the generating set $\{x_1, x_2\}$, and hence $B(h)$ is a graded double extension of the Manin's plane $\mathcal{O}_{-1}(\Bbbk) = A = \Bbbk_{-1}[x_1, x_2] = \Bbbk \{ x_1, x_2 \}/ \langle x_1x_2 + x_2x_1 \rangle$. Note that if we write $\sigma = \begin{bmatrix}
    \sigma_{11} & \sigma_{12} \\ \sigma_{21} & \sigma_{22}
\end{bmatrix}$, then
\begin{align*}
  \begin{bmatrix}
    \sigma_{11}(x_1) & \sigma_{12}(x_1) \\ \sigma_{21}(x_1) & \sigma_{22}(x_1) \end{bmatrix} = &\ h \begin{bmatrix} x_1 + x_2 & x_1 \\ x_2 & 0 \end{bmatrix}, \quad {\rm and} \\
    \begin{bmatrix}
     \sigma_{11}(x_1^2) & \sigma_{12}(x_1^2) \\ \sigma_{21}(x_1^2) & \sigma_{22}(x_1^2) \end{bmatrix} = &\ h^2 \begin{bmatrix} (x_1 + x_2)^2 + x_1x_2 & (x_1 + x_2) x_1 \\ x_2(x_1 + x_2) & x_2x_1 \end{bmatrix}, 
\end{align*}

which shows that $\sigma_{ij}(x_1^2) \neq \sigma_{ij}(x_1)^2$ for each pair $\sigma_{ij}$, and so each one of them is not an algebra homomorphism.

$B(h)$ is a remarkable algebra because it is neither an Ore extension nor a normal extension of an Artin-Schelter regular algebra (generated by three elements in degree one) of global dimension three. 

Theorem \ref{DifferentiallySmoothDOE} shows that $B(h)$ is not differentially smooth: in the relation $y_1 x_1 - h(x_1 y_1 + x_2y_1 + x_1y_2)=0$ the indeterminate $x_2$ only appears once, and so $B(h)$ satisfies (\ref{QuadraticrelationDSDOE}).
\end{example}

\begin{example}\label{DoubleExtensionsDSYorN}
Consider Example \ref{ZhangZhang2008Example4.1}. Although Theorem \ref{DifferentiallySmoothDOE} is formulated for polynomial extensions on four indeterminates, since the relations 
\begin{align*}
    y_1 x = &\ b^{-1}xy_2 + cx^2, \\
    y_1 x = &\ axy_1 + xy_2, \quad {\rm and} \\
    y_2 x = &\ (2 + 2b^{-1})xy_1 + xy_2
\end{align*}

define the algebras $B^2$, $B^3$ and $B^4$, respectively, and all of them satisfy the condition (\ref{QuadraticrelationDSDOE}), it follows that these three double extensions of $\Bbbk[t]$ are not differentially smooth.

The case of algebra $B^1$ is different and interesting. We know that by using the linear transformation $y_2 \to y_2 + \frac{bc}{b - 1}x$ we get that $c = 0$, whence $B^{1}(1, p, a, b, c) \cong B^{1}(1, p, a, b, 0)$. Let $p :=b^{-2}$.

Consider the maps of $B^1$ given by 
\begin{align}
   \nu_{x}(x) = &\ x, & \nu_{x}(y_1) = &\ by_1, & \nu_{x}(y_2) = &\ b^{-1}y_2, \notag \\ 
    \nu_{y_1}(x) = &\ b^{-1}x, & \nu_{y_1}(y_1) = &\ y_1, &  \nu_{y_1}(y_2) = &\ b^{-2}y_2, \notag \\
    \nu_{y_2}(x) = &\ bx, & \nu_{y_2}(y_1) = &\ b^2y_1, & \nu_{y_2}(y_2) = &\ y_2. \notag
\end{align}

The map $\nu_{x}$ can be extended to an algebra homomorphism of $B^1$ if and only if the definitions of $\nu_{x}(x)$, $\nu_{x}(y_1)$ and $\nu_{x}(y_2)$ respect relations  {\rm (}\ref{rel1B1}{\rm )}, {\rm (}\ref{rel2B1}{\rm )} and {\rm (}\ref{rel3B1}{\rm )}, i.e.
\begin{align*}
   \nu_{x}(y_2)\nu_{x}(y_1)-p\nu_{x}(y_1)\nu_{x}(y_2) = &\ \frac{bc}{1-b}(pb - 1)\nu_{x}(x)\nu_{x}(y_1) + a\nu_{x}(x^2), \\
   \nu_{x}(y_1)\nu_{x}(x)-b\nu_{x}(x)\nu_{x}(y_1) = &\ 0, \quad {\rm and} \\
  \nu_{x}(y_2)\nu_{x}(x)-b^{-1}\nu_{x}(x)\nu_{x}(y_2) = &\ c\nu_{x}(x^2).
\end{align*}

We obtain the equation 
\begin{align*}
     a^{-1}b\frac{bc}{1-b}(pb - 1)(b-1) = 0,
\end{align*}

which due to the conditions on the constants defining the algebra only holds when $c=0$.

Similarly, the map $\nu_{y_1}$ can be extended to an algebra homomorphism of $B^1$ if and only if the equality $p-b^{-2}=0$ is satisfied. 

Finally, by considering the extension of the map $\nu_{y_2}$ to an algebra homomorphism of $B^1$, we obtain again the condition $p-b^{-2}=0$.

Since we need to guarantee that
\begin{align*}\label{commuauto}
   \nu_{x} \circ \nu_{y_1} = &\ \nu_{y_1} \circ \nu_{x}, \\
   \nu_{x} \circ \nu_{y_2} = &\  \nu_{y_2} \circ \nu_{x}, \quad {\rm and} \\
   \nu_{y_2} \circ \nu_{y_1} = &\ \nu_{y_1} \circ \nu_{y_2}, 
\end{align*}

it is enough to satisfy these equalities for the generators $x$, $y_1$ and $y_2$, that is, 
\begin{align}
\nu_{x} \circ \nu_{y_1}(x) = &\ b^{-1}x,  \\
\nu_{y_1} \circ \nu_{x}(x) = &\ b^{-1}x,  \\ 
\nu_{x} \circ \nu_{y_1}(y_1) = &\ by_1, \\
\nu_{y_1} \circ \nu_{x}(y_1) = &\ by_1, \label{comp12} \\
 \nu_{x} \circ \nu_{y_1}(y_2) = &\ b^{-3}y_2, \quad {\rm and} \label{comp123} \\
 \nu_{y_1} \circ \nu_{x}(y_2) = &\ qb^{-3}y_2. \label{comp213}
\end{align}

As it is clear, composition $\nu_{x} \circ \nu_{y_1} = \nu_{y_1} \circ \nu_{x}$ is always satisfied.

Now, 
\begin{align}
    \nu_{x} \circ \nu_{y_2}(x) = &\ bx, \label{comp13}\\
    \nu_{y_2} \circ \nu_{x}(x) = &\ bx, \label{comp21} \\
    \nu_{x} \circ \nu_{y_2}(y_1) = &\ b^3y_1, \label{comp13'}\\
    \nu_{y_2} \circ \nu_{x}(y_1) = &\ b^3y_1, \label{comp22} \\
    \nu_{x} \circ \nu_{y_2}(y_2) = &\ b^{-1}y_2, \quad {\rm and} \label{comp22'} \\
    \nu_{y_2} \circ \nu_{x}(y_2) = &\ b^{-1}y_2. \label{comp23}
\end{align}

Again, we can see that composition $\nu_{x} \circ \nu_{y_2} = \nu_{y_2} \circ \nu_{x}$ is always satisfied.

Finally, 
\begin{align}
    \nu_{y_2} \circ \nu_{y_1}(x) = &\ x, \label{comp31'} \\
    \nu_{y_1} \circ \nu_{y_2}(x) = &\ x, \label{comp31} \\
    \nu_{y_2} \circ \nu_{y_1}(y_1) = &\ b^2y_1, \label{comp32'} \\
    \nu_{y_1} \circ \nu_{y_2}(y_1) = &\ b^2y_1,  \label{comp32} \\
    \nu_{y_2} \circ \nu_{y_1}(y_2) = &\ b^{-2}y_2, \quad {\rm and} \label{comp33'}\\
    \nu_{y_1} \circ \nu_{y_2}(y_2) = &\ b^{-2}y_2. \label{comp33}
\end{align}

Composition $\nu_{y_2} \circ \nu_{y_1} = \nu_{y_1} \circ \nu_{y_2}$ is always satisfied.

Consider $\Omega^{1}B^1$ a free right $B^1$-module of rank three with generators $dx$, $dy_1$ and $dy_2$. For all $q\in B^1$ define a left $b^1$-module structure by
\begin{align}
    qdx = &\ dx \nu_{x}(q), \notag \\ \quad qdy_1 = &\ dy_1\nu_{y_1}(q),\ {\rm and} \notag \\
    qdy_2 = &\ dy_2\nu_{y_2}(q) \label{relrightmod}.
\end{align}

The relations in $\Omega^{1}B^1$ are given by 
\begin{align}
xdx = &\ dx x, \notag \\
xdy_1 = &\ b^{-1}dy_1x, \notag \\
xdy_2 = &\ bdy_2x, \label{rel1} \\
y_1dx = &\ bdxy_1, \notag \\  
y_1dy_1 = &\ dy_1y_1, \notag \\
y_1dy_2 = &\ b^2dy_2y_1, \label{rel2} \\
y_2dx = &\ b^{-1}dxy_2, \notag \\
y_2dy_1 = &\ b^{-2}dy_1y_2, \ {\rm and} \notag \\
y_2dy_2 = &\ dy_2y_2. \label{rel3} 
\end{align}
    
We want to extend the correspondences 
\begin{equation*}
x \mapsto d x, \quad y_1 \mapsto d y_1 \quad {\rm and} \quad y_2\mapsto d y_2
\end{equation*} 

to a map $d: B^1 \to \Omega^{1} B^1$ satisfying the Leibniz rule. This is possible if it is compatible with the nontrivial relations {\rm (}\ref{rel1B1}{\rm )}, {\rm (}\ref{rel2B1}{\rm )} and {\rm (}\ref{rel3B1}{\rm )}, i.e. if the equalities
\begin{align*}
         dy_2 y_1 +y_2dy_1 = &\ pdy_1 y_2+ py_1dy_2 + adxx+axdx, \\
        dy_1 x +y_1dx = &\ bdxy_1 +bxdy_1, \quad {\rm and} \\
        dy_2 x +y_2dx = &\ b^{-1}dxy_2 +b^{-1}xdy_2.
\end{align*}

hold.

Define $\Bbbk$-linear maps 
\begin{equation*}
\partial_{x}, \partial_{y_1}, \partial_{y_2}: B^1 \rightarrow B^1
\end{equation*}

such that
\begin{align*}
    d(q)=dx\partial_{x}(q)+dy_1\partial_{y_1}(q)+dy_2\partial_{y_2}(q), \quad {\rm for\ all} \ q \in B^1.
\end{align*}

Since $dx$, $dy_1$ and $dy_2$ are free generators of the right $B^1$-module $\Omega^1B^1$, these maps are well-defined. Note that $d(q)=0$ if and only if $\partial_{x}(q)=\partial_{y_1}(q)=\partial_{y_2}(q)=0$. By using the three relations in {\rm (}\ref{relrightmod}{\rm )} and the definitions of the maps $\nu_{x}$, $\nu_{y_1}$ and $\nu_{y_2}$, we get that
\begin{align}
\partial_{x}(x^ky_1^ly_2^s) = &\ kx^{k-1}y_1^ly_2^s, \notag \\
\partial_{y_1}(x^ky_1^ly_2^s) = &\ lb^{-k}x^ky_1^{l-1}y_2^s, \quad {\rm and} \notag \\
\partial_{y_2}(x^ky_1^ly_2^s) = &\ sb^{k+2l}x^ky_1^ly_2^{s-1}.
\end{align}

Thus $d(q)=0$ if and only if $a$ is a scalar multiple of the identity. This shows that $(\Omega ^1,d)$ is connected where $\Omega B^1 = \Omega^0 B^1 \oplus \Omega^1 B^1 \oplus \Omega^2 B^1$.

The universal extension of $d$ to higher forms compatible with {\rm (}\ref{rel1}{\rm )}, {\rm (}\ref{rel2}{\rm )} and {\rm (}\ref{rel3}{\rm )} gives the following rules for $\Omega^2B^1$:
\begin{align}
dy_1\wedge dx = &\ -bdx\wedge dy_1, \\
dy_2\wedge dx = &\ -b^{-1}dx\wedge dy_2, \quad {\rm and} \\ 
dy_2\wedge dy_1 = &\ -b^{-2}dy_1\wedge dy_2 \label{relsecond}.
\end{align}

Since the automorphisms $\nu_{x}$, $\nu_{y_1}$ and $\nu_{y_2}$ commute with each other, there are no additional relationships to the previous ones, so we get the expression
\begin{align*}
    \Omega^2B^1 =  dx\wedge dy_1B^1\oplus dx\wedge dy_2B^1\oplus dy_1\wedge dy_2B^1.
\end{align*}

Since $\Omega^3B^1 = \omega B^1\cong B^1$ as a right and left $B^1$-module, with $\omega=dx\wedge dy_1 \wedge dy_2$, where $\nu_{\omega}=\nu_{x}\circ\nu_{y_1}\circ\nu_{y_2}$, we have that $\omega$ is a volume form of $B^1$. From Proposition \ref{BrzezinskiSitarz2017Lemmas2.6and2.7} (2) we get that $\omega$ is an integral form by setting
\begin{align*}
\omega_1^1  = &\ \bar{\omega}_1^1 = dx, \\  
\omega_2^1 = &\ \bar{\omega}_2^1 = dy_1, \\
\omega_3^1 = &\ \bar{\omega}_3^1 = dy_2, \\
    \omega_1^2 = &\ dy_1\wedge dy_2, \\
    \omega_2^2 = &\ b^{3}dx\wedge dy_1, \\ 
    \omega_3^2 = &\ -b^{-1}dx\wedge dy_2, \\
    \bar{\omega}_1^2 = &\ dy_1\wedge dy_2, \\
    \bar{\omega}_2^2 = &\ -b^{-1}dx\wedge dy_2, \quad {\rm and} \\ 
    \bar{\omega}_3^2 = &\ dx\wedge dy_1.
\end{align*}

By Proposition \ref{BrzezinskiSitarz2017Lemmas2.6and2.7} (2), we consider the expression $\omega' := dx\alpha + dy_1\beta + dy_2\gamma$ with $\alpha, \beta, \gamma \in \Bbbk$, to obtain the equality
\begin{align*}
    \sum_{i=1}^{3}\omega_{i}^{1}\pi_{\omega}(\bar{\omega}_i^{2}\wedge \omega') = &\ dx\pi_{\omega}(\alpha dy_1\wedge dy_2\wedge dx) + dy_1\pi_{\omega}(-\beta b^{2}dt\wedge  dy_2\wedge dy_1) \\
    & + dy_2\pi_{\omega}(\gamma dx\wedge dy_1\wedge dy_2) =  dx\alpha+dy_1\beta+dy_2\gamma = \omega'.
\end{align*}

On the other hand, if $\omega'' := dx\wedge dy_1\alpha+dx\wedge dy_2 \beta+dy_1\wedge dy_2 \gamma$ where $\alpha, \beta, \gamma \in \Bbbk$, we get that
\begin{align*}
\sum_{i=1}^{3}\omega_{i}^{2}\pi_{\omega}(\bar{\omega}_i^{1}\wedge \omega'') = &\  dy_1\wedge dy_2\pi_{\omega}(\gamma dx\wedge dy_1 \wedge dy_2) \\
&\ +b^{3}dx\wedge dy_1\pi_{\omega}(\alpha dy_2\wedge dx\wedge dy_1) \\
    &\ -b^{-1}dx\wedge dy_2\pi_{\omega}(\beta dy_1 \wedge dx \wedge dy_2) \\ 
    = &\ dx\wedge dy_1\alpha+dx\wedge dy_2 \beta+dy_1\wedge dy_2 \gamma \\
    = &\ \omega''.
\end{align*}

As we have seen above, all elements of different degrees can be generated by $\omega_i^j$ and $\bar{\omega}_i^{3-j}$  for $j=1,2$ and $i = 1 , 2, 3$, so Proposition \ref{BrzezinskiSitarz2017Lemmas2.6and2.7} (2) guarantees that $\omega$ is an integral form. Finally, Proposition \ref{integrableequiva} shows that $(\Omega B^1, d)$ is an integrable differential calculus of degree $3$, and since ${\rm GKdim}(B^1)$ is also $3$, it follows that $A$ is differentially smooth.
\end{example}

{\em Answer to the Question \ref{QuestionDSextendsDO}}. False. Consider the Manin's plane $\mathcal{O}_q(\Bbbk)$. As we know from Example \ref{Brzezinski2015ExamplesOE}, this algebra is differentially smooth, and it is well-known that ${\rm GKdim} (\mathcal{O}_q(\Bbbk)) = 2$. However, in Example \ref{ZhangZhang2008Example4.2} we saw that $B(h)$ is a double extension of $\mathcal{O}_q(\Bbbk)$, and so ${\rm GKdim} (B(h)) = 4$.

\section{Conclusiones and Future work}\label{FutureworkDSDoubleOreExtensions}

In Examples \ref{ZhangZhang2009Subcase4.1.1}, \ref{ZhangZhang2008Example4.2} and \ref{DoubleExtensionsDSYorN} we saw some interesting facts on the differential smoothness of double extensions of differentially smooth algebras. As we said in the Introduction, Ore extensions and normal extensions of regular algebras of dimension three were studied by Le Bruyn et al. \cite{LeBruynSmithVandenBergh1996}, and all of them were not considered by Zhang and Zhang \cite{ZhangZhang2009} in their classification of 26 families of regular algebras of type (14641). We think that a natural task is to investigate the differential smoothnes of the algebras appearing in \cite{LeBruynSmithVandenBergh1996}.

On the other hand, Artin-Schelter regular algebras of dimension at most two that are generated by elements of degree one are easy to classify; they have good properties similar to those of commutative polynomial rings as the following result shows.

\begin{proposition}[{\cite[Examples 1.8 and 1.9]{Rogalski2023}}]\label{Rogalski2023Examples1.8and1.9}
\begin{enumerate}
\item [\rm (1)] If $A$ is a $\Bbbk$-algebra which is regular of global dimension one, then $A$ is the commutative polynomial ring in one indeterminate, i.e. $A\cong \Bbbk[x]$.
    
\item [\rm (2)] Let $A$ be a regular algebra of global dimension two. Then either $A \cong \mathcal{O}_q(\Bbbk) = \Bbbk \{x, y\} / \langle yx - qxy \rangle$ for some $q\in \Bbbk^{*}$, or $A\cong \mathcal{J}(\Bbbk) = \Bbbk\{x, y\}/ \langle yx - xy - y^2\rangle$. 
\end{enumerate}
\end{proposition}

From Example \ref{Brzezinski2015ExamplesOE} (iii) we know that $\Bbbk[x]$, $\mathcal{O}_q(\Bbbk)$ and $\mathcal{J}(\Bbbk)$ are differentially smooth.

Since Artin-Schelter regular algebras of dimension three that are generated by elements of degree one have been classified by Artin, Schelter, Tate and Van den Bergh \cite{ArtinSchelter1987, ArtinTateVandenBergh2007, ArtinTateVandenBergh1991}), and that the characterization of these algebras requires greater mathematical techniques that have not been considered at the time when this paper was written, the study of the differential smoothness of these algebras will be one of our immediate tasks.

\begin{landscape}
\begin{table}[h]
\caption{Double extensions}
\label{FirstTableDOE}
\begin{center}
\resizebox{20cm}{!}{
\setlength\extrarowheight{11pt}
\begin{tabular}{ |c|c|c|c|c|c| } 
\hline
Double extension & Relations defining the double extension  & $\Sigma_{ij}$ & $M_{ij}$ & Data $\{\Sigma, M, P, Q\}$ & Conditions \\
\hline
\multirow{3}{*}{$\mathbb{A}$} &  $x_2x_1 = x_1 x_2,\quad y_2y_1 = y_1y_2+y_{1}^{2}$, &  &  &  & \\ 
& $y_1x_1 = x_1y_1,  \quad y_1x_2=x_2y_1+x_1y_2$, & $\Sigma_{11}=\begin{bmatrix}
    1 & 0  \\ 0 & 1 
\end{bmatrix}, \quad \Sigma_{12}=\begin{bmatrix}
    0 & 0  \\ 1 & 0 
\end{bmatrix}, \quad \Sigma_{21}=\begin{bmatrix}
    0 & 0  \\ 0 & -2 
\end{bmatrix}, \quad \Sigma_{22}=\begin{bmatrix}
    1 & 0  \\ -1 & 1 
\end{bmatrix}$ & $M_{11}=\begin{bmatrix}
    1 & 0  \\ 0 & 1 
\end{bmatrix}, \quad M_{12}=\begin{bmatrix}
    0 & 0  \\ 0 & 0 
\end{bmatrix}, \quad M_{21}=\begin{bmatrix}
    0 & 1  \\ 0 & -1 
\end{bmatrix}, \quad M_{22}=\begin{bmatrix}
    1 & 0  \\ -2 & 1 
\end{bmatrix}$ & $\Sigma=\begin{bmatrix}
    1 & 0 & 0 & 0 \\ 0 & 1 & 1 & 0 \\ 0 & 0 & 1 & 0 \\ 0 & -2 & -1 & 1
\end{bmatrix}, \quad M=\begin{bmatrix}
    1 & 0 & 0 & 0 \\ 0 & 1 & 0 & 0 \\ 0 & 1 & 1 & 0 \\ 0 & -1 & -2 & 1
\end{bmatrix}, \quad P=(1, 1), \quad Q=(1, 0)$  &  \\ 
&   $y_2x_1 = x_1y_2, \quad y_2x_2=-2x_2y_1-x_1y_2+x_2y_2$ & &  &  & \\
\hline
\multirow{3}{*}{$\mathbb{B}$} &   $x_2x_1 = px_1 x_2, \quad y_2y_1 = py_1y_2$, & & & & \\ 
& $y_1x_1 = x_2y_2,  \quad y_1x_2=x_1y_2$,  & $\Sigma_{11}=\begin{bmatrix}
    0 & 0  \\ 0 & 0
\end{bmatrix}, \quad \Sigma_{12}=\begin{bmatrix}
    0 & 1  \\ 1 & 0 
\end{bmatrix}, \quad \Sigma_{21}=\begin{bmatrix}
    0 & -1  \\ 1 & 0
\end{bmatrix}, \quad \Sigma_{22}=\begin{bmatrix}
    0 & 0  \\ 0 & 0 
\end{bmatrix}$ & $M_{11}=\begin{bmatrix}
    0 & 0  \\ 0 & 0 
\end{bmatrix}, \quad M_{12}=\begin{bmatrix}
    0 & 1  \\ -1 & 0 
\end{bmatrix}, \quad M_{21}=\begin{bmatrix}
    0 & 1  \\ 1 & 0
\end{bmatrix}, \quad M_{22}=\begin{bmatrix}
    0 & 0  \\ 0 & 0 
\end{bmatrix}$  & $\Sigma=\begin{bmatrix}
    0 & 0 & 0 & 1 \\ 0 & 0 & 1 & 0 \\ 0 & -1 & 0 & 0 \\ 1 & 0 & 0 & 0
\end{bmatrix}, \quad M=\begin{bmatrix}
    0 & 0 & 0 & 1 \\ 0 & 0 & -1 & 0 \\ 0 & 1 & 0 & 0 \\ 1 & 0 & 0 & 0
\end{bmatrix}, \quad P=(p, 0), \quad Q=(p, 0)$  & $p^2 = -1$ \\ 
&    $y_2x_1 = -x_2y_1, \quad y_2x_2=x_1y_1$ & & & & \\
\hline
\multirow{3}{*}{$\mathbb{C}$} &  $x_2x_1 = px_1 x_2, \quad y_2y_1 = py_1y_2$, & & &  & \\ 
&  $y_1x_1 = -x_1y_1+p^2x_2y_1+x_1y_2-px_2y_2,  \quad y_1x_2 =-px_1y_1+x_2y_1+x_1y_2-px_2y_2$,  & $\Sigma_{11}=\begin{bmatrix}
    -1 & p^2  \\ -p & 1 
\end{bmatrix}, \quad \Sigma_{12}=\begin{bmatrix}
    1 & -p  \\ 1 & -p 
\end{bmatrix}, \quad \Sigma_{21}=\begin{bmatrix}
    -p & -2p^2  \\ -p & p^2 
\end{bmatrix}, \quad \Sigma_{22}=\begin{bmatrix}
    p & -p  \\ 1 & -1 
\end{bmatrix}$ & $M_{11}=\begin{bmatrix}
    -1 & 1  \\ -p & p 
\end{bmatrix}, \quad M_{12}=\begin{bmatrix}
    p^2 & -p  \\ -2p^2 & -p 
\end{bmatrix}, \quad M_{21}=\begin{bmatrix}
    -p & 1  \\ -p & 1 
\end{bmatrix}, \quad M_{22}=\begin{bmatrix}
    1 & -p  \\ p^2 & -1 
\end{bmatrix}$ &  $\Sigma=\begin{bmatrix}
    -1 & p^2 & 1 & -p \\ -p & 1 & 1 & -p \\ -p & -2p^2 & p & -p \\ -p & p^2 & 1 & -1
\end{bmatrix}, \quad M=\begin{bmatrix}
    -1 & 1 & p^2 & -p \\ -p & p & -2p^2 & -p \\ -p & 1 & 1 & -p \\ -p & 1 & p^2 & -1
\end{bmatrix}, \quad P=(p, 0), \quad Q=(p, 0)$  & $p^2+p+1 = 0$ \\ 
&    $y_2x_1 = -px_1y_1-2p^2x_2y_1+px_1y_2-px_2y_2, \quad y_2x_2 =-px_1y_1+p^2x_2y_1+x_1y_2-x_2y_2$ & & & & \\
\hline
\multirow{3}{*}{$\mathbb{D}$} &  $x_2x_1 = -x_1 x_2, \quad y_2y_1 = py_1y_2$, & & &  &  \\ 
& $y_1x_1 = -px_1y_1,  \quad y_1x_2=-p^2x_2y_1+x_1y_2$,  & $\Sigma_{11}=\begin{bmatrix}
    -p & 0  \\ 0 & -p^2 
\end{bmatrix}, \quad \Sigma_{12}=\begin{bmatrix}
    0 & 0  \\ 1 & 0 
\end{bmatrix}, \quad \Sigma_{21}=\begin{bmatrix}
    0 & 0  \\ 1 & 0 
\end{bmatrix}, \quad \Sigma_{22}=\begin{bmatrix}
    p & 0  \\ 0 & 1 
\end{bmatrix}$ & $M_{11}=\begin{bmatrix}
    -p & 0  \\ 0 & p 
\end{bmatrix}, \quad M_{12}=\begin{bmatrix}
    0 & 0  \\ 0 & 0 
\end{bmatrix}, \quad M_{21}=\begin{bmatrix}
    0 & 1  \\ 1 & 0 
\end{bmatrix}, \quad M_{22}=\begin{bmatrix}
    -p^2 & 0  \\ 0 & 1 
\end{bmatrix}$  & $\Sigma=\begin{bmatrix}
    -p & 0 & 0 & 0 \\ 0 & -p^2 & 1 & 0 \\ 0 & 0 & p & 0 \\ 1 & 0 & 0 & 1
\end{bmatrix}, \quad M=\begin{bmatrix}
    -p & 0 & 0 & 0 \\ 0 & p & 0 & 0 \\ 0 & 1 & -p^2 & 0 \\ 1 & 0 & 0 & 1
\end{bmatrix}, \quad P=(p, 0), \quad Q=(-1, 0)$  & $p\in \{-1,1\}$ \\ 
&    $y_2x_1 = px_1y_2, \quad y_2x_2=x_1y_1+x_2y_2$ &  & & & \\
\hline
\multirow{3}{*}{$\mathbb{E}$} &  $x_2x_1 = -x_1 x_2, \quad y_2y_1 = py_1y_2$, & & &  & \\ 
&  $y_1x_1 = x_1y_2+x_2y_2,  \quad y_1x_2=x_1y_2-x_2y_2$, & $\Sigma_{11}=\begin{bmatrix}
    0 & 0  \\ 0 & 0 
\end{bmatrix}, \quad \Sigma_{12}=\begin{bmatrix}
    1 & 1  \\ 1 & -1 
\end{bmatrix}, \quad \Sigma_{21}=\begin{bmatrix}
    -1 & 1  \\ 1 & 1 
\end{bmatrix}, \quad \Sigma_{22}=\begin{bmatrix}
    0 & 0  \\ 0 & 0 
\end{bmatrix}$ & $M_{11}=\begin{bmatrix}
    0 & 1  \\ -1 & 0 
\end{bmatrix}, \quad M_{12}=\begin{bmatrix}
    0 & 1  \\ 1 & 0 
\end{bmatrix}, \quad M_{21}=\begin{bmatrix}
    0 & 1  \\ 1 & 0 
\end{bmatrix}, \quad M_{22}=\begin{bmatrix}
    0 & -1  \\ 1 & 0 
\end{bmatrix}$ & $\Sigma=\begin{bmatrix}
    0 & 0 & 1 & 1 \\ 0 & 0 & 1 & -1 \\ -1 & 1 & 0 & 0 \\ 1 & 1 & 0 & 0
\end{bmatrix}, \quad M=\begin{bmatrix}
    0 & 1 & 0 & 1 \\ -1 & 0 & 1 & 0 \\ 0 & 1 & 0 & -1 \\ 1 & 0 & 1 & 0
\end{bmatrix}, \quad P=(p, 0), \quad Q=(-1, 0)$  & $p^2=-1$ \\ 
&    $y_2x_1 = -x_1y_1+x_2y_1, \quad y_2x_2=x_1y_1+x_2y_1$ & & & &  \\
\hline
\multirow{3}{*}{$\mathbb{F}$} &  $x_2x_1 = -x_1 x_2, \quad y_2y_1 = py_1y_2$, & & &  & \\ 
&  $y_1x_1 = -x_1y_1-px_2y_1+x_1y_2-x_2y_2,  \quad y_1x_2 =-px_1y_1+x_2y_1+x_1y_2+x_2y_2$, & $\Sigma_{11}=\begin{bmatrix}
    -1 & -p  \\ -p & 1 
\end{bmatrix}, \quad \Sigma_{12}=\begin{bmatrix}
    1 & -1  \\ 1 & 1 
\end{bmatrix}, \quad \Sigma_{21}=\begin{bmatrix}
    -p & p  \\ -p & -p 
\end{bmatrix}, \quad \Sigma_{22}=\begin{bmatrix}
    p & 1  \\ 1 & -p 
\end{bmatrix}$ & $M_{11}=\begin{bmatrix}
    +1 & 1  \\ -p & p 
\end{bmatrix}, \quad M_{12}=\begin{bmatrix}
    -p & -1  \\ p & 1 
\end{bmatrix}, \quad M_{21}=\begin{bmatrix}
    -p & 1  \\ -p & 1 
\end{bmatrix}, \quad M_{22}=\begin{bmatrix}
    1 & 1  \\ -p & -p 
\end{bmatrix}$ & $\Sigma=\begin{bmatrix}
    -1 & -p & 1 & -1 \\ -p & 1 & 1 & 1 \\ -p & p & p & 1 \\ -p & -p & 1 & -p
\end{bmatrix}, \quad M=\begin{bmatrix}
    -1 & 1 & -p & -1 \\ -p & p & p & 1 \\ -p & 1 & 1 & 1 \\ -p & 1 & -p & -p
\end{bmatrix}, \quad P=(p, 0), \quad Q=(-1, 0)$  & $p^2=-1$ \\ 
&    $y_2x_1 = -px_1y_1+px_2y_1+px_1y_2+x_2y_2, \quad y_2x_2 =-px_1y_1-px_2y_1+x_1y_2-px_2y_2$ & & &  & \\
\hline
\multirow{3}{*}{$\mathbb{G}$} &   $ x_2x_1 = x_1 x_2, \quad y_2y_1 = py_1y_2$, & & &  & \\ 
&  $y_1x_1 = px_1y_1,  \quad y_1x_2=px_1y_1+p^2x_2y_1+x_1y_2$, & $\Sigma_{11}=\begin{bmatrix}
    p & 0  \\ p & p^2 
\end{bmatrix}, \quad \Sigma_{12}=\begin{bmatrix}
    0 & 0  \\ 1 & 0 
\end{bmatrix}, \quad \Sigma_{21}=\begin{bmatrix}
    0 & 0  \\ f & 0 
\end{bmatrix}, \quad \Sigma_{22}=\begin{bmatrix}
    p & 0  \\ -1 & 1 
\end{bmatrix}$ & $M_{11}=\begin{bmatrix}
    p & 0  \\ 0 & p 
\end{bmatrix}, \quad M_{12}=\begin{bmatrix}
    0 & 0  \\ 0 & 0 
\end{bmatrix}, \quad M_{21}=\begin{bmatrix}
    p & 1  \\ f & -1 
\end{bmatrix}, \quad M_{22}=\begin{bmatrix}
    p^2 & 0  \\ 0 & 1 
\end{bmatrix}$  &  $\Sigma=\begin{bmatrix}
    p & 0 & 0 & 0 \\ p & p^2 & 1 & 0 \\ 0 & 0 & p & 0 \\ f & 0 & -1 & 1
\end{bmatrix}, \quad M=\begin{bmatrix}
    p & 0 & 0 & 0 \\ 0 & p & 0 & 0 \\ p & 1 & p^2 & 0 \\ f & -1 & 0 & 1
\end{bmatrix}, \quad P=(p, 0), \quad Q=(1, 0)$  & $p\not = 0, \pm 1$ and $f \not = 0$ \\ 
&    $y_2x_1 = px_1y_2, \quad y_2x_2=fx_1y_1-x_1y_2+x_2y_2$ & & & & \\
\hline
\end{tabular}
}
\end{center}
\end{table}
\end{landscape}

\begin{landscape}
{\huge{
\begin{table}[h]
\caption{Double extensions}
\label{SecondTableDOE}
\begin{center}
\resizebox{20cm}{!}{
\setlength\extrarowheight{11pt}
\begin{tabular}{ |c|c|c|c|c|c| } 
\hline
Double extension & Relations defining the double extension  & $\Sigma_{ij}$ & $M_{ij}$ & Data $\{\Sigma, M, P, Q\}$ & Conditions \\
\hline
\multirow{3}{*}{$\mathbb{H}$} &   $ x_2x_1 = x_1 x_2+x_1^2, \quad y_2y_1 = -y_1y_2$, &  & &  &\\ 
&  $y_1x_1 = x_1y_2,  \quad y_1x_2=fx_1y_2+x_2y_2$, & $\Sigma_{11}=\begin{bmatrix}
    0 & 0  \\ 0 & 0 
\end{bmatrix}, \quad \Sigma_{12}=\begin{bmatrix}
    1 & 0  \\ f & 1 
\end{bmatrix}, \quad \Sigma_{21}=\begin{bmatrix}
    1 & 0  \\ f & 1 
\end{bmatrix}, \quad \Sigma_{22}=\begin{bmatrix}
    0 & 0  \\ 0 & 0 
\end{bmatrix}$ & $M_{11}=\begin{bmatrix}
    0 & 1  \\ 1 & 0 
\end{bmatrix}, \quad M_{12}=\begin{bmatrix}
    0 & 0  \\ 0 & 0 
\end{bmatrix}, \quad M_{21}=\begin{bmatrix}
    0 & f  \\ f & 0 
\end{bmatrix}, \quad M_{22}=\begin{bmatrix}
    0 & 1  \\ 1 & 0 
\end{bmatrix}$ & $\Sigma=\begin{bmatrix}
    0 & 0 & 1 & 0 \\ 0 & 0 & f & 1 \\ 1 & 0 & 0 & 0 \\ f & 1 & 0 & 0
\end{bmatrix}, \quad M=\begin{bmatrix}
    0 & 1 & 0 & 0 \\ 1 & 0 & 0 & 0 \\ 0 & f & 0 & 1 \\ f & 0 & 1 & 0
\end{bmatrix}, \quad P=(-1, 0), \quad Q=(1, 1)$  & $f \not = 0$ \\ 
&   $y_2x_1 = x_1y_1, \quad y_2x_2=fx_1y_1+x_2y_1$ & &  & & \\
\hline
\multirow{3}{*}{$\mathbb{I}$} &  $ x_2x_1 = qx_1 x_2, \quad y_2y_1 = -y_1y_2$, & &  &  & \\ 
& $y_1x_1 = -qx_1y_1-qx_2y_1+x_1y_2-qx_2y_2,  \quad y_1x_2 = x_1y_1+x_2y_1+x_1y_2-qx_2y_2$,  & $\Sigma_{11}=\begin{bmatrix}
    -q & -q  \\ 1 & 1 
\end{bmatrix}, \quad \Sigma_{12}=\begin{bmatrix}
    1 & -q  \\ 1 & -q 
\end{bmatrix}, \quad \Sigma_{21}=\begin{bmatrix}
    1 & q  \\ -1 & -q 
\end{bmatrix}, \quad \Sigma_{22}=\begin{bmatrix}
    q & -q  \\ 1 & -1 
\end{bmatrix}$ & $M_{11}=\begin{bmatrix}
    -q & 1  \\ 1 & q 
\end{bmatrix}, \quad M_{12}=\begin{bmatrix}
    -q & -q  \\ q & -q 
\end{bmatrix}, \quad M_{21}=\begin{bmatrix}
    1 & 1  \\ -1 & 1 
\end{bmatrix}, \quad M_{22}=\begin{bmatrix}
    1 & -q  \\ -q & -1 
\end{bmatrix}$  & $\Sigma=\begin{bmatrix}
    -q & -q & 1 & -q \\ 1 & 1 & 1 & -q \\ 1 & q & q & -q \\ -1 & -q & 1 & -1
\end{bmatrix}, \quad M=\begin{bmatrix}
    -q & 1 & -q & -q \\ 1 & q & q & -q \\ 1 & 1 & 1 & -q \\ -1 & 1 & -q & -1
\end{bmatrix}, \quad P=(-1, 0), \quad Q=(q, 0)$  & $q^2=-1$ \\ 
&    $y_2x_1 = x_1y_1+qx_2y_1+qx_1y_2-qx_2y_2, \quad y_2x_2 =-x_1y_1-qx_2y_1+x_1y_2-x_2y_2$ & & & & \\
\hline
\multirow{3}{*}{$\mathbb{J}$} &  $ x_2x_1 = qx_1 x_2, \quad y_2y_1 = -y_1y_2$, & & & &  \\ 
&  $y_1x_1 = x_2y_1+x_2y_2,  \quad y_1x_2=-x_1y_1+x_1y_2$, & $\Sigma_{11}=\begin{bmatrix}
    0 & 1  \\ -1 & 0 
\end{bmatrix}, \quad \Sigma_{12}=\begin{bmatrix}
    0 & 1  \\ 1 & 0 
\end{bmatrix}, \quad \Sigma_{21}=\begin{bmatrix}
    0 & 1  \\ 1 & 0 
\end{bmatrix}, \quad \Sigma_{22}=\begin{bmatrix}
    0 & -1  \\ 1 & 0 
\end{bmatrix}$ & $M_{11}=\begin{bmatrix}
    0 & 0  \\ 0 & 0 
\end{bmatrix}, \quad M_{12}=\begin{bmatrix}
    1 & 1  \\ 1 & -1 
\end{bmatrix}, \quad M_{21}=\begin{bmatrix}
    -1 & 1  \\ 1 & 1 
\end{bmatrix}, \quad M_{22}=\begin{bmatrix}
    0 & 0  \\ 0 & 0 
\end{bmatrix}$  &  $\Sigma=\begin{bmatrix}
    0 & 1 & 0 & 1 \\ -1 & 0 & 1 & 0 \\ 0 & 1 & 0 & -1 \\ 1 & 0 & 1 & 0
\end{bmatrix}, \quad M=\begin{bmatrix}
    0 & 0 & 1 & 1 \\ 0 & 0 & 1 & -1 \\ -1 & 1 & 0 & 0 \\ 1 & 1 & 0 & 0
\end{bmatrix}, \quad P=(-1, 0), \quad Q=(q, 0)$  & $q^2=-1$ \\ 
&    $y_2x_1 = x_2y_1-x_2y_2, \quad y_2x_2=x_1y_1+x_1y_2$ &  & &  & \\
\hline
\multirow{3}{*}{$\mathbb{K}$} &  $ x_2x_1 = qx_1 x_2, \quad y_2y_1 = -y_1y_2$, &  & & & \\ 
&  $y_1x_1 = x_1y_1,  \quad y_1x_2=x_2y_2$, & $\Sigma_{11}=\begin{bmatrix}
    1 & 0  \\ 0 & 0 
\end{bmatrix}, \quad \Sigma_{12}=\begin{bmatrix}
    0 & 0  \\ 0 & 1 
\end{bmatrix}, \quad \Sigma_{21}=\begin{bmatrix}
    0 & 0  \\ 0 & f 
\end{bmatrix}, \quad \Sigma_{22}=\begin{bmatrix}
    1 & 0  \\ 0 & 0 
\end{bmatrix}$ & $M_{11}=\begin{bmatrix}
    1 & 0  \\ 0 & 1 
\end{bmatrix}, \quad M_{12}=\begin{bmatrix}
    0 & 0  \\ 0 & 0 
\end{bmatrix}, \quad M_{21}=\begin{bmatrix}
    0 & 0  \\ 0 & 0 
\end{bmatrix}, \quad M_{22}=\begin{bmatrix}
    0 & 1  \\ f & 0 
\end{bmatrix}$ & $\Sigma=\begin{bmatrix}
    1 & 0 & 0 & 0 \\ 0 & 0 & 0 & 1 \\ 0 & 0 & 1 & 0 \\ 0 & f & 0 & 0
\end{bmatrix}, \quad M=\begin{bmatrix}
    1 & 0 & 0 & 0 \\ 0 & 1 & 0 & 0 \\ 0 & 0 & 0 & 1 \\ 0 & 0 & f & 0
\end{bmatrix}, \quad P=(-1, 0), \quad Q=(q, 0)$  & $q\in \{-1,1\}$ and $f \not = 0$ \\ 
&    $y_2x_1 = x_1y_2, \quad y_2x_2=fx_2y_1$ & & & & \\
\hline
\multirow{3}{*}{$\mathbb{L}$} &  $ x_2x_1 = qx_1 x_2, \quad y_2y_1 = -y_1y_2$, &  & &  & \\ 
&  $y_1x_1 = fx_1y_2,  \quad y_1x_2=x_2y_2$, & $\Sigma_{11}=\begin{bmatrix}
    0 & 0  \\ 0 & 0 
\end{bmatrix}, \quad \Sigma_{12}=\begin{bmatrix}
    f & 0  \\ 0 & 1 
\end{bmatrix}, \quad \Sigma_{21}=\begin{bmatrix}
    f & 0  \\ 0 & 1 
\end{bmatrix}, \quad \Sigma_{22}=\begin{bmatrix}
    0 & 0  \\ 0 & 0 
\end{bmatrix}$ & $M_{11}=\begin{bmatrix}
    0 & f  \\ f & 0 
\end{bmatrix}, \quad M_{12}=\begin{bmatrix}
    0 & 0  \\ 0 & 0 
\end{bmatrix}, \quad M_{21}=\begin{bmatrix}
    0 & 0  \\ 0 & 0 
\end{bmatrix}, \quad M_{22}=\begin{bmatrix}
    0 & 1  \\ 1 & 0 
\end{bmatrix}$ & $\Sigma=\begin{bmatrix}
    0 & 0 & f & 0 \\ 0 & 0 & 0 & 1 \\ f & 0 & 0 & 0 \\ 0 & 1 & 0 & 0
\end{bmatrix}, \quad M=\begin{bmatrix}
    0 & f & 0 & 0 \\ f & 0 & 0 & 0 \\ 0 & 0 & 0 & 1 \\ 0 & 0 & 1 & 0
\end{bmatrix}, \quad P=(-1, 0), \quad Q=(q, 0)$  & $q\in \{-1,1\}$ and $f \not = 0$ \\ 
&    $y_2x_1 = fx_1y_1, \quad y_2x_2=x_2y_1$ & & & & \\
\hline
\multirow{3}{*}{$\mathbb{M}$} & $ x_2x_1 = -x_1 x_2, \quad y_2y_1 = -y_1y_2$, &  & &  & \\ 
& $y_1x_1 = x_2y_1+x_1y_2,  \quad y_1x_2=fx_1y_1-x_2y_2$,  & $\Sigma_{11}=\begin{bmatrix}
    0 & 1  \\ f & 0 
\end{bmatrix}, \quad \Sigma_{12}=\begin{bmatrix}
    1 & 0  \\ 0 & -1 
\end{bmatrix}, \quad \Sigma_{21}=\begin{bmatrix}
    1 & 0  \\ 0 & -1 
\end{bmatrix}, \quad \Sigma_{22}=\begin{bmatrix}
    0 & -1  \\ -f & 0 
\end{bmatrix}$ & $M_{11}=\begin{bmatrix}
    0 & 1  \\ 1 & 0 
\end{bmatrix}, \quad M_{12}=\begin{bmatrix}
    1 & 0  \\ 0 & -1 
\end{bmatrix}, \quad M_{21}=\begin{bmatrix}
    f & 0  \\ 0 & -f 
\end{bmatrix}, \quad M_{22}=\begin{bmatrix}
    0 & -1  \\ -1 & 0 
\end{bmatrix}$ &  $\Sigma=\begin{bmatrix}
    0 & 1 & 1 & 0 \\ f & 0 & 0 & -1 \\ 1 & 0 & 0 & -1 \\ 0 & -1 & -f & 0
\end{bmatrix}, \quad M=\begin{bmatrix}
    0 & 1 & 1 & 0 \\ 1 & 0 & 0 & -1 \\ f & 0 & 0 & -1 \\ 0 & -f & -1 & 0
\end{bmatrix}, \quad P=(-1, 0), \quad Q=(-1, 0)$  & $f \not = 1$ \\ 
&   $y_2x_1 = x_1y_1-x_2y_2, \quad y_2x_2=-x_2y_1-fx_1y_2$ &  & & & \\
\hline
\multirow{3}{*}{$\mathbb{N}$} &  $ x_2x_1 = -x_1 x_2, \quad y_2y_1 = -y_1y_2$, & & &  & \\ 
&  $y_1x_1 = -gx_2y_1+fx_2y_2,  \quad y_1x_2=gx_1y_1+fx_1y_2$, & $\Sigma_{11}=\begin{bmatrix}
    0 & -g  \\ g & 0 
\end{bmatrix}, \quad \Sigma_{12}=\begin{bmatrix}
    0 & f  \\ f & 0 
\end{bmatrix}, \quad \Sigma_{21}=\begin{bmatrix}
    0 & f  \\ f & 0 
\end{bmatrix}, \quad \Sigma_{22}=\begin{bmatrix}
    0 & -g  \\ g & 0 
\end{bmatrix}$ & $M_{11}=\begin{bmatrix}
    0 & 0  \\ 0 & 0 
\end{bmatrix}, \quad M_{12}=\begin{bmatrix}
    -g & f  \\ f & -g 
\end{bmatrix}, \quad M_{21}=\begin{bmatrix}
    g & f  \\ f & g 
\end{bmatrix}, \quad M_{22}=\begin{bmatrix}
    0 & 0  \\ 0 & 0 
\end{bmatrix}$  &  $\Sigma=\begin{bmatrix}
    0 & -g & 0 & f \\ g & 0 & f & 0 \\ 0 & f & 0 & -g \\ f & 0 & g & 0
\end{bmatrix}, \quad M=\begin{bmatrix}
    1 & 0 & -g & f \\ 0 & 0 & f & -g \\ g & f & 0 & 0 \\ f & g & 0 & 0
\end{bmatrix}, \quad P=(-1, 0), \quad Q=(-1, 0)$  & $f^2 \not = g^2$ \\ 
&    $y_2x_1 = fx_2y_1-gx_2y_2, \quad y_2x_2=fx_1y_1+gx_1y_2$ & & & &
\\
\hline
\end{tabular}
}
\end{center}
\end{table}
}}
\end{landscape}

\begin{landscape}
{\huge{
\begin{table}[h]
\caption{Double extensions}
\label{ThirdTableDOE}
\begin{center}
\resizebox{20cm}{!}{
\setlength\extrarowheight{7pt}
\begin{tabular}{ |c|c|c|c|c|c| } 
\hline
Double extension & Relations defining the double extension  & $\Sigma_{ij}$ & $M_{ij}$ & Data $\{\Sigma, M, P, Q\}$ & Conditions
\\
\hline
\multirow{3}{*}{$\mathbb{O}$} &   $ x_2x_1 = -x_1 x_2, \quad y_2y_1 = -y_1y_2$, & & & & \\ 
&  $y_1x_1 = x_1y_1+fx_2y_2,  \quad y_1x_2=-x_2y_1+x_1y_2$, & $\Sigma_{11}=\begin{bmatrix}
    1 & 0  \\ 0 & -1 
\end{bmatrix}, \quad \Sigma_{12}=\begin{bmatrix}
    0 & f  \\ 1 & 0 
\end{bmatrix}, \quad \Sigma_{21}=\begin{bmatrix}
    0 & f  \\ 1 & 0 
\end{bmatrix}, \quad \Sigma_{22}=\begin{bmatrix}
    -1 & 0  \\ 0 & 1 
\end{bmatrix}$ & $M_{11}=\begin{bmatrix}
    1 & 0  \\ 0 & -1 
\end{bmatrix}, \quad M_{12}=\begin{bmatrix}
    0 & f  \\ f & 0 
\end{bmatrix}, \quad M_{21}=\begin{bmatrix}
    0 & 1  \\ 1 & 0 
\end{bmatrix}, \quad M_{22}=\begin{bmatrix}
    -1 & 0  \\ 0 & 1 
\end{bmatrix}$  &  $\Sigma=\begin{bmatrix}
    1 & 0 & 0 & f \\ 0 & -1 & 1 & 0 \\ 0 & f & -1 & 0 \\ 1 & 0 & 0 & 1
\end{bmatrix}, \quad M=\begin{bmatrix}
    1 & 0 & 0 & f \\ 0 & -1 & f & 0 \\ 0 & 1 & -1 & 0 \\ 1 & 0 & 0 & 1
\end{bmatrix}, \quad P=(-1, 0), \quad Q=(-1, 0)$  & $f \not = -1$ \\ 
&  $y_2x_1 = fx_2y_1-x_1y_2, \quad y_2x_2=x_1y_1+x_2y_2$ &  &  & &   \\
\hline
\multirow{3}{*}{$\mathbb{P}$} &  $ x_2x_1 = -x_1 x_2, \quad y_2y_1 = -y_1y_2$, &  & & & \\ 
& $y_1x_1 = x_1y_2+fx_2y_2,  \quad y_1x_2=x_1y_2+x_2y_2$,  & $\Sigma_{11}=\begin{bmatrix}
    0 & 0  \\ 0 & 0 
\end{bmatrix}, \quad \Sigma_{12}=\begin{bmatrix}
    1 & f  \\ 1 & 1 
\end{bmatrix}, \quad \Sigma_{21}=\begin{bmatrix}
    1 & -f  \\ -1 & 1 
\end{bmatrix}, \quad \Sigma_{22}=\begin{bmatrix}
    0 & 0  \\ 0 & 0 
\end{bmatrix}$ & $M_{11}=\begin{bmatrix}
    0 & 1  \\ 1 & 0 
\end{bmatrix}, \quad M_{12}=\begin{bmatrix}
    0 & f  \\ -f & 0 
\end{bmatrix}, \quad M_{21}=\begin{bmatrix}
    0 & 1  \\ -1 & 0 
\end{bmatrix}, \quad M_{22}=\begin{bmatrix}
    0 & 1  \\ 1 & 0 
\end{bmatrix}$  &  $\Sigma=\begin{bmatrix}
    0 & 0 & 1 & f \\ 0 & 0 & 1 & 1 \\ 1 & -f & 0 & 0 \\ -1 & 1 & 0 & 0
\end{bmatrix}, \quad M=\begin{bmatrix}
    0 & 1 & 0 & f \\ 1 & 0 & -f & 0 \\ 0 & 1 & 0 & 1 \\ -1 & 0 & 1 & 0
\end{bmatrix}, \quad P=(-1, 0), \quad Q=(-1, 0)$  & $f \not = -1$ \\ 
&    $y_2x_1 = x_1y_1-fx_2y_1, \quad y_2x_2=-x_1y_1+x_2y_1$ &  & & & \\
\hline
\multirow{3}{*}{$\mathbb{Q}$} &  $x_2x_1 = -x_1 x_2,\quad y_2y_1 = -y_1y_2$, & & &  & \\ 
&  $y_1x_1 = x_1y_2,  \quad y_1x_2=x_1y_1+x_2y_1+x_1y_2$, & $\Sigma_{11}=\begin{bmatrix}
    0 & 0  \\ 1 & 1 
\end{bmatrix}, \quad \Sigma_{12}=\begin{bmatrix}
    1 & 0  \\ 1 & 0 
\end{bmatrix}, \quad \Sigma_{21}=\begin{bmatrix}
    -1 & 0  \\ 1 & 0 
\end{bmatrix}, \quad \Sigma_{22}=\begin{bmatrix}
    0 & 0  \\ -1 & 1 
\end{bmatrix}$ & $M_{11}=\begin{bmatrix}
    0 & 1  \\ -1 & 0 
\end{bmatrix}, \quad M_{12}=\begin{bmatrix}
    0 & 0  \\ 0 & 0 
\end{bmatrix}, \quad M_{21}=\begin{bmatrix}
    1 & 1  \\ 1 & -1 
\end{bmatrix}, \quad M_{22}=\begin{bmatrix}
    1 & 0  \\ 0 & 1 
\end{bmatrix}$  & $\Sigma=\begin{bmatrix}
    0 & 0 & 1 & 0 \\ 1 & 1 & 1 & 0 \\ -1 & 0 & 0 & 0 \\ 1 & 0 & -1 & 1
\end{bmatrix}, \quad M=\begin{bmatrix}
    0 & 1 & 0 & 0 \\ -1 & 0 & 0 & 0 \\ 1 & 1 & 1 & 0 \\ 1 & -1 & 0 & 1
\end{bmatrix}, \quad P=(-1, 0), \quad Q=(-1, 0)$  &  \\ 
&    $y_2x_1 = -x_1y_1, \quad y_2x_2=x_1y_1-x_1y_2+x_2y_2$ & & & & \\
\hline
\multirow{3}{*}{$\mathbb{R}$} &  $x_2x_1 = -x_1 x_2,\quad y_2y_1 = -y_1y_2$, & & &  & \\ 
&  $y_1x_1 = x_1y_1+x_2y_1+x_1y_2,  \quad y_1x_2=x_1y_2$, & $\Sigma_{11}=\begin{bmatrix}
    1 & 1  \\ 0 & 0 
\end{bmatrix}, \quad \Sigma_{12}=\begin{bmatrix}
    1 & 0  \\ 1 & 0 
\end{bmatrix}, \quad \Sigma_{21}=\begin{bmatrix}
    0 & 1  \\ 0 & -1 
\end{bmatrix}, \quad \Sigma_{22}=\begin{bmatrix}
    0 & 0  \\ -1 & 1 
\end{bmatrix}$ & $M_{11}=\begin{bmatrix}
    1 & 1  \\ 0 & 0 
\end{bmatrix}, \quad M_{12}=\begin{bmatrix}
    1 & 0  \\ 1 & 0 
\end{bmatrix}, \quad M_{21}=\begin{bmatrix}
    0 & 1  \\ 0 & -1 
\end{bmatrix}, \quad M_{22}=\begin{bmatrix}
    0 & 0  \\ -1 & 1 
\end{bmatrix}$ & $\Sigma=\begin{bmatrix}
    1 & 1 & 1 & 0 \\ 0 & 0 & 1 & 0 \\ 0 & 1 & 0 & 0 \\ 0 & -1 & -1 & 1
\end{bmatrix}, \quad M=\begin{bmatrix}
    1 & 1 & 1 & 0 \\ 0 & 0 & 1 & 0 \\ 0 & 1 & 0 & 0 \\ 0 & -1 & -1 & 1
\end{bmatrix}, \quad P=(-1, 0), \quad Q=(-1, 0)$  &  \\ 
&    $y_2x_1 = x_2y_1, \quad y_2x_2=-x_2y_1-x_1y_2+x_2y_2$ & & & & \\
\hline
\multirow{3}{*}{$\mathbb{S}$} &  $x_2x_1 = -x_1 x_2,\quad y_2y_1 = -y_1y_2$, & & &  & \\ 
& $y_1x_1 = -x_1y_1+x_2y_1+x_1y_2+x_2y_2,  \quad y_1x_2 = x_1y_1-x_2y_1+x_1y_2+x_2y_2$,  & $\Sigma_{11}=\begin{bmatrix}
    -1 & 1  \\ 1 & -1 
\end{bmatrix}, \quad \Sigma_{12}=\begin{bmatrix}
    1 & 1  \\ 1 & 1 
\end{bmatrix}, \quad \Sigma_{21}=\begin{bmatrix}
    1 & 1  \\ 1 & 1 
\end{bmatrix}, \quad \Sigma_{22}=\begin{bmatrix}
    -1 & 1  \\ 1 & -1 
\end{bmatrix}$ & $M_{11}=\begin{bmatrix}
    -1 & 1  \\ 1 & -1 
\end{bmatrix}, \quad M_{12}=\begin{bmatrix}
    1 & 1  \\ 1 & 1 
\end{bmatrix}, \quad M_{21}=\begin{bmatrix}
    1 & 1  \\ 1 & 1 
\end{bmatrix}, \quad M_{22}=\begin{bmatrix}
    -1 & 1  \\ 1 & -1 
\end{bmatrix}$  &  $\Sigma=\begin{bmatrix}
    -1 & 1 & 1 & 1 \\ 1 & -1 & 1 & 1 \\ 1 & 1 & -1 & 1 \\ 1 & 1 & 1 & -1
\end{bmatrix}, \quad M=\begin{bmatrix}
    -1 & 1 & 1 & 1 \\ 1 & -1 & 1 & 1 \\ 1 & 1 & -1 & 1 \\ 1 & 1 & 1 & -1
\end{bmatrix}, \quad P=(-1, 0), \quad Q=(-1, 0)$  &  \\ 
&    $y_2x_1 = x_1y_1+x_2y_1-x_1y_2+x_2y_2, \quad y_2x_2 =x_1y_1+x_2y_1+x_1y_2-x_2y_2$ & & & &  \\
\hline
\multirow{3}{*}{$\mathbb{T}$} &  $x_2x_1 = -x_1 x_2,\quad y_2y_1 = -y_1y_2$, & & &  & \\ 
& $y_1x_1 = -x_1y_1+x_2y_1+x_1y_2+x_2y_2,  \quad y_1x_2 = x_1y_1-x_2y_1+x_1y_2+x_2y_2$,  & $\Sigma_{11}=\begin{bmatrix}
    -1 & 1  \\ 1 & -1
\end{bmatrix}, \quad \Sigma_{12}=\begin{bmatrix}
    1 & 1  \\ 1 & 1 
\end{bmatrix}, \quad \Sigma_{21}=\begin{bmatrix}
    1 & 1  \\ 1 & 1 
\end{bmatrix}, \quad \Sigma_{22}=\begin{bmatrix}
    1 & -1  \\ -1 & 1 
\end{bmatrix}$ & $M_{11}=\begin{bmatrix}
    -1 & 1  \\ 1 & 1 
\end{bmatrix}, \quad M_{12}=\begin{bmatrix}
    1 & 1  \\ 1 & -1 
\end{bmatrix}, \quad M_{21}=\begin{bmatrix}
    1 & 1  \\ 1 & -1 
\end{bmatrix}, \quad M_{22}=\begin{bmatrix}
    -1 & 1  \\ 1 & 1 
\end{bmatrix}$  & $\Sigma=\begin{bmatrix}
    -1 & 1 & 1 & 1 \\ 1 & -1 & 1 & 1 \\ 1 & 1 & 1 & -1 \\ 1 & 1 & -1 & 1
\end{bmatrix}, \quad M=\begin{bmatrix}
    -1 & 1 & 1 & 1 \\ 1 & 1 & 1 & -1 \\ 1 & 1 & -1 & 1 \\ 1 & -1 & 1 & 1
\end{bmatrix}, \quad P=(-1, 0), \quad Q=(-1, 0)$  &  \\ 
&   $y_2x_1 = x_1y_1+x_2y_1+x_1y_2-x_2y_2, \quad y_2x_2 =x_1y_1+x_2y_1-x_1y_2+x_2y_2$ & &  & & \\
\hline
\end{tabular}
}
\end{center}
\end{table}
}}
\end{landscape}

\begin{landscape}
{\huge{
\begin{table}[h]
\caption{Double extensions}
\label{FourthTableDOE}
\begin{center}
\resizebox{20cm}{!}{
\setlength\extrarowheight{7pt}
\begin{tabular}{ |c|c|c|c|c|c| } 
\hline
Double extension & Relations defining the double extension  & $\Sigma_{ij}$ & $M_{ij}$ & Data $\{\Sigma, M, P, Q\}$ & Conditions
\\
\hline
\multirow{3}{*}{$\mathbb{U}$} &  $x_2x_1 = -x_1 x_2,\quad y_2y_1 = -y_1y_2$, &  & & & \\ 
&  $y_1x_1 = -x_1y_1+x_2y_1+x_1y_2+x_2y_2,  \quad y_1x_2 = x_1y_1+x_2y_1+x_1y_2-x_2y_2$, & $\Sigma_{11}=\begin{bmatrix}
    -1 & 1  \\ 1 & 1 
\end{bmatrix}, \quad \Sigma_{12}=\begin{bmatrix}
    1 & 1  \\ 1 & -1 
\end{bmatrix}, \quad \Sigma_{21}=\begin{bmatrix}
    1 & 1  \\ 1 & -1 
\end{bmatrix}, \quad \Sigma_{22}=\begin{bmatrix}
    -1 & 1  \\ 1 & 1 
\end{bmatrix}$ & $M_{11}=\begin{bmatrix}
    -1 & 1  \\ 1 & -1 
\end{bmatrix}, \quad M_{12}=\begin{bmatrix}
    1 & 1  \\ 1 & 1 
\end{bmatrix}, \quad M_{21}=\begin{bmatrix}
    1 & 1  \\ 1 & 1 
\end{bmatrix}, \quad M_{22}=\begin{bmatrix}
    1 & -1  \\ -1 & 1 
\end{bmatrix}$   & $\Sigma=\begin{bmatrix}
    -1 & 1 & 1 & 1 \\ 1 & 1 & 1 & -1 \\ 1 & 1 & -1 & 1 \\ 1 & -1 & 1 & 1
\end{bmatrix}, \quad M=\begin{bmatrix}
    -1 & 1 & 1 & 1 \\ 1 & -1 & 1 & 1 \\ 1 & 1 & 1 & -1 \\ 1 & 1 & -1 & 1
\end{bmatrix}, \quad P=(-1, 0), \quad Q=(-1, 0)$  &  \\ 
&    $y_2x_1 = x_1y_1+x_2y_1-x_1y_2+x_2y_2, \quad y_2x_2 =x_1y_1-x_2y_1+x_1y_2+x_2y_2$ & & & &  \\
\hline
\multirow{3}{*}{$\mathbb{V}$} & $x_2x_1 = x_1 x_2,\quad y_2y_1 = -y_1y_2$, & & &  &  \\ 
& $y_1x_1 = x_2y_1+x_1y_2,  \quad y_1x_2= x_2y_1$,  & $\Sigma_{11}=\begin{bmatrix}
    0 & 1  \\ 0 & 1 
\end{bmatrix}, \quad \Sigma_{12}=\begin{bmatrix}
    1 & 0  \\ 0 & 0 
\end{bmatrix}, \quad \Sigma_{21}=\begin{bmatrix}
    -1 & 1  \\ 0 & 0 
\end{bmatrix}, \quad \Sigma_{22}=\begin{bmatrix}
    0 & 0  \\ 0 & 1 
\end{bmatrix}$ & $M_{11}=\begin{bmatrix}
    0 & 1  \\ -1 & 0 
\end{bmatrix}, \quad M_{12}=\begin{bmatrix}
    1 & 0  \\ 1 & 0 
\end{bmatrix}, \quad M_{21}=\begin{bmatrix}
    0 & 0  \\ 0 & 0 
\end{bmatrix}, \quad M_{22}=\begin{bmatrix}
    1 & 0  \\ 0 & 1 
\end{bmatrix}$  & $\Sigma=\begin{bmatrix}
    0 & 1 & 1 & 0 \\ 0 & 1 & 0 & 0 \\ -1 & 1 & 0 & 0 \\ 0 & 0 & 0 & 1
\end{bmatrix}, \quad M=\begin{bmatrix}
    0 & 1 & 1 & 0 \\ -1 & 0 & 1 & 0 \\ 0 & 0 & 1 & 0 \\ 0 & 0 & 0 & 1
\end{bmatrix}, \quad P=(-1, 0), \quad Q=(1, 0)$  &  \\ 
&    $y_2x_1 = -x_1y_1+x_2y_1, \quad y_2x_2=x_2y_2$ & & & & \\
\hline
\multirow{3}{*}{$\mathbb{W}$} &  $x_2x_1 = x_1 x_2,\quad y_2y_1 = -y_1y_2$, & & &  & \\ 
&  $y_1x_1 = fx_2y_1+x_1y_2,  \quad y_1x_2= x_1y_1-x_2y_2$, & $\Sigma_{11}=\begin{bmatrix}
    0 & f  \\ 1 & 0 
\end{bmatrix}, \quad \Sigma_{12}=\begin{bmatrix}
    1 & 0  \\ 0 & -1 
\end{bmatrix}, \quad \Sigma_{21}=\begin{bmatrix}
    1 & 0  \\ 0 & -1 
\end{bmatrix}, \quad \Sigma_{22}=\begin{bmatrix}
    0 & f  \\ 1 & 0 
\end{bmatrix}$ & $M_{11}=\begin{bmatrix}
    0 & 1  \\ 1 & 0 
\end{bmatrix}, \quad M_{12}=\begin{bmatrix}
    f & 0  \\ 0 & f 
\end{bmatrix}, \quad M_{21}=\begin{bmatrix}
    1 & 0  \\ 0 & 1 
\end{bmatrix}, \quad M_{22}=\begin{bmatrix}
    0 & -1  \\ -1 & 0 
\end{bmatrix}$ &  $\Sigma=\begin{bmatrix}
    0 & f & 1 & 0 \\ 1 & 0 & 0 & -1 \\ 1 & 0 & 0 & f \\ 0 & -1 & 1 & 0
\end{bmatrix}, \quad M=\begin{bmatrix}
    0 & 1 & f & 0 \\ 1 & 0 & 0 & f \\ 1 & 0 & 0 & -1 \\ 0 & 1 & -1 & 0
\end{bmatrix}, \quad P=(-1, 0), \quad Q=(1, 0)$  & $f \not = -1$ \\ 
&    $y_2x_1 = x_1y_1+fx_2y_2, \quad y_2x_2=-x_2y_1+x_1y_2$ & & & & \\
\hline
\multirow{3}{*}{$\mathbb{X}$} &   $x_2x_1 = x_1 x_2,\quad y_2y_1 = -y_1y_2$, & & & & \\ 
&  $y_1x_1 = x_1y_2,  \quad y_1x_2= x_1y_2+x_2y_2$, & $\Sigma_{11}=\begin{bmatrix}
    0 & 0  \\ 0 & 0 
\end{bmatrix}, \quad \Sigma_{12}=\begin{bmatrix}
    1 & 0  \\ 1 & 1 
\end{bmatrix}, \quad \Sigma_{21}=\begin{bmatrix}
    1 & 0  \\ 1 & 1 
\end{bmatrix}, \quad \Sigma_{22}=\begin{bmatrix}
    0 & 0  \\ 0 & 0 
\end{bmatrix}$ & $M_{11}=\begin{bmatrix}
    0 & 1  \\ 1 & 0 
\end{bmatrix}, \quad M_{12}=\begin{bmatrix}
    0 & 0  \\ 0 & 0 
\end{bmatrix}, \quad M_{21}=\begin{bmatrix}
    1 & 1  \\ 1 & 0 
\end{bmatrix}, \quad M_{22}=\begin{bmatrix}
    0 & 1  \\ 1 & 0 
\end{bmatrix}$ & $\Sigma=\begin{bmatrix}
    0 & 0 & 1 & 0 \\ 0 & 0 & 1 & 1 \\ 1 & 0 & 0 & 0 \\ 1 & 1 & 0 & 0
\end{bmatrix}, \quad M=\begin{bmatrix}
    0 & 1 & 0 & 0 \\ 1 & 0 & 0 & 0 \\ 1 & 1 & 0 & 1 \\ 1 & 0 & 1 & 0
\end{bmatrix}, \quad P=(-1, 0), \quad Q=(1, 0)$  &  \\ 
&    $y_2x_1 = x_1y_1, \quad y_2x_2=x_1y_1+x_2y_1$ & & & & \\
\hline
\multirow{3}{*}{$\mathbb{Y}$} &   $x_2x_1 = x_1 x_2,\quad y_2y_1 = -y_1y_2$, & & & & \\ 
& $y_1x_1 = x_1y_1,  \quad y_1x_2= fx_1y_1-x_2y_1+x_1y_2$,  & $\Sigma_{11}=\begin{bmatrix}
    1 & 0  \\ f & -1 
\end{bmatrix}, \quad \Sigma_{12}=\begin{bmatrix}
    0 & 0  \\ 1 & 0 
\end{bmatrix}, \quad \Sigma_{21}=\begin{bmatrix}
    0 & 0  \\ 1 & 0 
\end{bmatrix}, \quad \Sigma_{22}=\begin{bmatrix}
    1 & 0  \\ f & -1 
\end{bmatrix}$ & $M_{11}=\begin{bmatrix}
    1 & 0  \\ 0 & 1 
\end{bmatrix}, \quad M_{12}=\begin{bmatrix}
    0 & 0  \\ 0 & 0 
\end{bmatrix}, \quad M_{21}=\begin{bmatrix}
    f & 1  \\ 1 & f 
\end{bmatrix}, \quad M_{22}=\begin{bmatrix}
    -1 & 0  \\ 0 & -1 
\end{bmatrix}$  & $\Sigma=\begin{bmatrix}
    1 & 0 & 0 & 0 \\ f & -1 & 1 & 0 \\ 0 & 0 & 1 & 0 \\ 1 & 0 & f & -1
\end{bmatrix}, \quad M=\begin{bmatrix}
    1 & 0 & 0 & 0 \\ 0 & 1 & 0 & 0 \\ f & 1 & -1 & 0 \\ 1 & f & 0 & -1
\end{bmatrix}, \quad P=(-1, 0), \quad Q=(1, 0)$  & $f$ is general \\ 
&    $ y_2x_1 = x_1y_2, \quad y_2x_2=x_1y_1+fx_1y_2-x_2y_2$ & & & & \\
\hline
\multirow{3}{*}{$\mathbb{Z}$} &  $x_2x_1 = -x_1 x_2,\quad y_2y_1 = y_1y_2$, & & &  & \\ 
& $y_1x_1 = x_1y_1+x_2y_2,  \quad y_1x_2 = x_2y_1+x_1y_2$,  & $\Sigma_{11}=\begin{bmatrix}
    1 & 0  \\ 0 & 1 
\end{bmatrix}, \quad \Sigma_{12}=\begin{bmatrix}
    0 & 1  \\ 1 & 0
\end{bmatrix}, \quad \Sigma_{21}=\begin{bmatrix}
    0 & f  \\ f & 0 
\end{bmatrix}, \quad \Sigma_{22}=\begin{bmatrix}
    -1 & 0  \\ 0 & -1 
\end{bmatrix}$ & $M_{11}=\begin{bmatrix}
    1 & 0  \\ 0 & -1 
\end{bmatrix}, \quad M_{12}=\begin{bmatrix}
    0 & 1  \\ f & 0 
\end{bmatrix}, \quad M_{21}=\begin{bmatrix}
    0 & 1  \\ f & 0 
\end{bmatrix}, \quad M_{22}=\begin{bmatrix}
    1 & 0  \\ 0 & -1 
\end{bmatrix}$  & $\Sigma=\begin{bmatrix}
    1 & 0 & 0 & 1 \\ 0 & 1 & 1 & 0 \\ 0 & f & -1 & 0 \\ f & 0 & 0 & -1
\end{bmatrix}, \quad M=\begin{bmatrix}
    1 & 0 & 0 & 1 \\ 0 & -1 & f & 0 \\ 0 & 1 & 1 & 0 \\ f & 0 & 0 & -1
\end{bmatrix}, \quad P=(1, 0), \quad Q=(-1, 0)$  & $f(1+f)\not = 0$ is general \\ 
&    $y_2x_1 = fx_2y_1-x_1y_2, \quad y_2x_2=fx_1y_1-x_2y_2$ & & & & \\
\hline
\end{tabular}
}
\end{center}
\end{table}
}}

\begin{remark} Zhang and Zhang \cite[Subcase 4.4.4]{ZhangZhang2009} formulated the relations of the algebra $\mathbb{Z}$. However, the relations shown contain some typos since the coefficients of the relations do not match the entries of the matrix $\Sigma$. Our version of these relations are presented in Table \ref{FourthTableDOE}. They wrote the relations $y_1x_1 = x_1y_2+x_2y_2$ and $y_2x_2=fx_1y_2-x_2y_2$, and the typo concerns that they considered the first factor $x_1y_2$ when it should be $x_1y_1$ according to matrix $\Sigma$.
\end{remark}
\end{landscape}

\end{document}